\documentclass[a4paper]{article}

\usepackage[english]{babel}
\usepackage[utf8x]{inputenc}
\usepackage{geometry} 
\usepackage{euscript}
\usepackage{xcolor}
\usepackage{tikz}
\usetikzlibrary{angles, quotes}
\usepackage{url}

\usepackage{amsrefs} 

\usepackage{graphicx} 
\usepackage{epstopdf}
\usepackage{amscd}
\usepackage{array} 
\usepackage{verbatim} 
\usepackage{amssymb}
\usepackage{amsmath}
\usepackage{amsthm}
\usepackage{csquotes}
\usepackage{mathtools}

\newtheorem{exmp}{Example}
\newtheorem{definition}{Definition}
\newtheorem{remark}{Remark}
\newtheorem{theorem}{Theorem}
\newtheorem*{theorem*}{Theorem}
\newtheorem{lemma}{Lemma}
\newtheorem{prop}{Proposition}

\newtheorem{cor}{Corollary}
\newtheorem{question}{Question}

\def\cone{K}
\def\latt{\mathbb{Z}K}
\def\rcone{\mathbb{R}_{\ge0}K}
\def\rlatt{\mathbb{R}K}

\def\rp{\Lambda^+}
\def\mathu{\mathcal{C}}
\DeclareMathOperator{\Tra}{Tr}

\DeclareMathOperator{\gcone}{Cone}

\DeclareMathOperator{\Int}{Int}

\DeclareMathOperator{\rk}{rk}

\DeclareMathOperator{\Irr}{Irr}

\DeclareMathOperator{\Aut}{Aut}
\DeclareMathOperator{\Cat}{Cat}

\title{Stability of amalgamated free products and HNN extensions}

\author{Maria Gerasimova \and
  Konstantin Shchepin}
\newcommand{\Addresses}{{
  \bigskip
  \footnotesize
  \medskip
  Maria Gerasimova, \textsc{University of M{\"u}nster, Mathematical Institute, Einsteinstrasse 62, 48149, M{\"u}nster, Germany}\par\nopagebreak
  \textit{E-mail address} : \texttt{mari9gerasimova@mail.ru}
  
\medskip

  Konstantin Shchepin \textsc{}\par\nopagebreak
  \textit{E-mail address} : \texttt{shchepin.konstantin@gmail.com}
}}

\begin{document}
\vspace{-8ex}
\date{}
\maketitle

\begin{abstract}
We study stability of amalgamated free products and HNN extensions of stable groups over finite groups. We focus on operator norm stability, Hilbert-Schmidt stability and stability in permutations. We provide many new examples of stable (or flexibly stable) non-amenable groups.
\end{abstract}

\section{Introduction}

Given a system of equations $R$, one can ask if it is ``stable'' meaning that each ``almost'' solution is ``close'' to an actual solution. We consider a set $S$ of noncommutative variables and a set of equations in the form of words from the free group $F_S$ generated by $S$. It turns out that stability of a system of equations $R$ is a property of a group $G=\langle S|R\rangle$.
This means that if two sets of relations give isomorphic groups, then either both of them are stable or not stable. 

To give a precise meaning to ``almost'' and ``close'' one needs to fix a class of groups with the bi-invariant metric. Different classes of groups give different definitions of stability. Let us briefly describe the classes which we will be interested in throughout the paper. For us the main examples will be finite-dimensional unitary groups with a metric given by some norm (more precisely, by the operator norm or by the normalized Hilbert-Schmidt norm) and finite symmetric groups equipped with the normalized Hamming distance. We will denote the corresponding classes of groups by $\mathu_{op}$, $\mathu_{HS}$ and $\mathu_P$. 

There also exists a flexible version of stability. Roughly speaking, it means that one can look for a homomorphism to a slightly larger group from the same class $\mathu$. One needs to define flexible stability for each class of groups $\mathu$ separately. This notion seems to be strictly weaker than being stable, although for now we do not have any example showing this (for instance, for amenable groups stability and flexible stability coincide, see \cite{ioana2020stability} for $\mathu_P$ and \cite{eckhardt2023amenable} for $\mathu_{HS}$). We will discuss the precise definitions in Section 2.3. 

Operator norm stability seems to be quite a rare property. For example, by the famous result of Voiculescu, $\mathbb{Z}^2$ is not $\mathu_{op}$-stable (\cite{voiculescu1983asymptotically}).  
Although there are some non-trivial positive results (\cite{eilers2020c}), there are many more negative ones (\cite{eilers2020c}, \cite{dadarlat2021obstructions}, \cite{bader2023stability}). 

Hilbert-Schmidt stability and stability in permutations look much friendlier. For example, in the case of amenable groups there are nice characterisations of stability (see \cite{hadwin2018stability} for $\mathu_{HS}$ and \cite{becker2019stability} for $\mathu_P$) that allow to produce many interesting examples (see \cite{becker2019stability}, \cite{levit2022infinitely}, \cite{zheng2019rigid}, \cite{levit2022uncountably}  for $\mathu_P$ and \cite{eckhardt2023amenable}, \cite{levit2022characters} for $\mathu_{HS}$). 

Unfortunately, in the setting of non-amenable groups no necessary and sufficient conditions are known. Because of this all non-amenable examples have to be examined separately. It is trivial both for $\mathu_{HS}$ and for $\mathu_P$ that free products of stable groups are again stable. Non-trivial examples of $P$-stable groups include virtually free groups (\cite{lazarovich2021virtually}). Also surface groups are known to be flexibly $P$-stable (\cite{lazarovich2019surface}). 
Non-amenable examples of ${HS}$-stable groups include certain graph product groups (\cite{atkinson2018some}), one-relator groups with non-trivial center (\cite{hadwin2018stability}) and virtually free groups (\cite{ger2021}).

The results about stability of virtually free groups in particular imply that amalgamated free products and HNN extensions of finite groups are stable. This leads to a natural question if one can extend these results to amalgamated free products and HNN extensions of stable groups. Note that in general there is no hope that amalgamated free products of stable groups will be stable (even if we amalgamate over a finite index subgroup, see \cite{becker2020group} for $\mathu_P$). However, stability of virtually free groups both in $\mathu_P$ and $\mathu_{HS}$ suggests that one can try to look at amalgamated free products and HNN extensions over finite groups. In this paper we study this question. 

Let us briefly describe our results. 
Operator norm stability appears to be the nicest one for amalgamated free products and HNN extensions. The reason for this is that representations of a finite group that are close in operator norm are conjugated (by an operator that is close to identity). The following result, that was probably overlooked, gives us more examples of operator norm stable groups.  

\begin{theorem*}[Theorem~\ref{operator_amalgam}, Theorem~\ref{operator_HNN}]
    Amalgamated free products and HNN extensions of operator norm stable groups over finite subgroups are operator norm stable.
\end{theorem*}

For $\mathu_{HS}$ and $\mathu_P$ we present a uniform approach that allows us to write only one statement and only one proof for both types of stability. In the further results $\mathu$ denotes either $\mathu_{HS}$, or $\mathu_P$ and the statements hold for both cases.

Stability of amalgamated free products and HNN extensions for $\mathu_{HS}$ and $\mathu_P$ has another layer of complexity -- close representations and actions of a finite group now are not necessarily conjugated. 
The only case where the same argument works without any additional changes is the special case of HNN extensions. 

\begin{theorem*}[Theorem~\ref{double_HNN}]
    An HNN extension of a $\mathu$-stable group over two identical finite subgroups is $\mathu$-stable.
\end{theorem*}

Even the special case of amalgamated free products (so-called group doubles) requires some non-trivial technical results on subsemigroups of $\mathbb{Z}^N$. In this case our methods allow us to prove only flexible stability (even if we assume that our group is stable and not just flexibly stable). 

\begin{theorem*}[Theorem~\ref{double_amalgam}]
    A group double of a (flexibly) $\mathu$-stable group over a finite subgroup is flexibly $\mathu$-stable.
\end{theorem*}

To prove (flexible) stability in more general case we need to study properties of restrictions of representations and actions to finite subgroups. 
Surprisingly, restrictions of representations and actions to finite almost normal subgroups behave similarly. This allows us to get the following results.

\begin{theorem*}[Theorem~\ref{amalgam_flex}, Theorem~\ref{hnn_flex}]
    Amalgamated free products and HNN extensions of (flexibly) $\mathu$-stable groups over finite almost normal subgroups are flexibly $\mathu$-stable.
\end{theorem*}

In some cases we can prove stability and not just flexible stability.
For a stable group we introduce the notion of a ``decomposing map''. Such maps essentially encode some parameters of a representation or an action of a stable group that can be changed arbitrarily with a small change of a representation or an action, respectively. This notion allows us to prove the following results.  

\begin{theorem*}[Theorem~\ref{amalgam_stability}, Theorem~\ref{hnn_stability}]
    Amalgamated free products and HNN extensions of $\mathu$-stable groups over finite normal subgroups are $\mathu$-stable.
\end{theorem*}

Let us also mention that in the recent paper \cite{arzhantseva2023constraint} Arzhantseva and P{\u{a}}unescu studied constraint stability and introduced the notion of action trace. This allowed them to find some $P$-stable groups among amalgamated free products over finite groups. We remark that our methods are vastly different and produce different results.
In Section 7.2 we discuss their approach and give some comments on the similar approach for $\mathu_{HS}$.

\subsection*{Acknowledgements}

During this research M.G.\ was supported by the DFG -- Project-ID 427320536 -- SFB 1442, and under Germany's Excellence Strategy EXC 2044 390685587, Mathematics M{\"u}nster: Dynamics--Geometry--Structure.

\section{Preliminaries}
\subsection{Stability of relations and groups}

Let us fix a class of metric groups $\mathu=\{\mathbf{C}\}$, where each group $\mathbf{C}\in\mathu$ is equipped with a fixed bi-invariant metric $d=d^\mathbf{C}$. 
All other groups that we consider will be countable and finitely generated by default.

\begin{definition}
    We will say that a group $G$ is $\mathu$-stable if for any sequence of maps
    $$\varphi_n \colon G \to \mathbf{C}_{n}\in\mathu \quad \text{satisfying}\quad
    \forall g,h\in G \quad d(\varphi_n(g)\cdot\varphi_n(h),\varphi_n(gh))\xrightarrow[n\to\infty]{} 0 $$
    there exists a sequence of homomorphisms
    $$\psi_n\colon G\to \mathbf{C}_{n}\quad \text{such that} \quad 
    \forall g\in G \quad d(\varphi_n(g),\psi_n(g))\xrightarrow[n\to\infty]{} 0.$$
\end{definition}

Different classes of metric groups $\mathu$ give rise to the different notions of stability. Let us define the classes we will be most interested in.  

\begin{enumerate}
    \item $\mathu_{op}=\{(\Aut(X),||\cdot||_{op}): X \text{ is a finite-dimensional Hilbert space}\}$ -- the class of finite-dimensional unitary groups with the operator norm (operator norm stability). The metric on unitary groups comes from the norm by the following formula
    $$d(A_1,A_2)=||A_1-A_2||_{op}.$$
    \item $\mathu_{HS}=\{(\Aut(X),||\cdot||_{HS}): X \text{ is a finite-dimensional Hilbert space}\}$ -- the class of finite-dimensional unitary groups with Hilbert-Schmidt norm (Hilbert-Schmidt stability, $HS$-stability or $\mathu_{HS}$-stability)
    $$||A||_{HS}=\Tra(A^*A)^{\frac{1}{2}},$$
    where $\Tra$ is the normalized trace. 
    \item $\mathu_{P}=\{(\Aut(X),d^{Hamm}): X \text{ is a finite set}\}$ -- the class of finite symmetric groups with the normalized Hamming distance (stability in permutations, $P$-stability or $\mathu_P$-stability)
    $$d^{Hamm}(\sigma_1,\sigma_2) = \frac{1}{|X|}\#\{x\in X: \sigma_1(x)\ne \sigma_2(x)\}.$$
\end{enumerate}

The notation for these classes of groups are intentionally made similar. We want to emphasize the fact that finite-dimensional unitary groups and finite symmetric groups are just automorphism groups of objects in the categories $\Cat_{HS}:=FdHilb$ (of finite-dimensional Hilbert spaces) and $\Cat_P:=FinSet$ (of finite sets), respectively. 

Very often our proofs for $\mathu_{HS}$ and $\mathu_P$ will be essentially the same. This motivated us to use the generic terminology and similar notations for these two cases. Throughout the paper we will use many notations that may have either $HS$ or $P$ as a lower index. If there are no indices in some statement, then this statement works both for $HS$-stability and $P$-stability. 

 For a group $\overline{G}$, any two homomorphisms $\varphi_1,\varphi_2\colon \overline{G}\to \Aut(X)\in\mathu$ and a finite set $B\subset \overline{G}$ we will define the distance between $\varphi_1$ and $\varphi_2$ over $B$ as
$$d_B(\varphi_1,\varphi_2):=\max_{b\in B} d(\varphi_1(b),\varphi_2(b)).$$
We will abuse the notations and write $1_X$ for the trivial homomorphism $1_X\colon\overline{G}\to \Aut(X)$ and for an identity automorphism $1_X\in\Aut(X)$. We will also make no difference between homomorphisms from a quotient group $\overline{G}/\langle\langle R\rangle\rangle$ and homomorphisms from $\overline{G}$ that map $R$ to identity. The following defintion of a stable epimorphism was introduced for epimorphisms given by finite sets of relations in \cite{lazarovich2021virtually}.

\begin{definition}\label{def_st_epi}
    Let us consider a group $\overline{G}$ generated by a finite set $S$. 
    We will say that an epimorphism $\pi\colon\overline{G}\to G$ is $\mathu$-stable if for any $\varepsilon>0$ there exist $\delta>0$ and a finite set $E\subseteq \ker \pi$ such that for any $\varphi\colon \overline{G}\to \Aut(X)\in\mathu$ with $d_E(\varphi, 1_X)\le\delta$ there exists $\psi\colon G\to \Aut(X)$ with $d_S(\varphi,\psi)\le\varepsilon$.
\end{definition}

The equivalent definition reformulation of $P$-stability for not necessarily finitely generated groups was given in \cite{ioana2020stability} (Lemma 3.1). The proof of equivalence works without any changes for stability with respect to any class of metric groups $\mathu$.  
Using Definition~\ref{def_st_epi} we can state this result for finitely generated groups as following.

\begin{prop}
    The group $G$ generated by a finite set $S$ is $\mathu$-stable if and only if the epimorphism $\pi\colon F_S\to G$ is $\mathu$-stable. 
\end{prop}

\subsection{Cones associated with groups}
A coproduct operation in the categories $\Cat_{HS}$ and $\Cat_{P}$ is a direct sum of Hilbert spaces and a disjoint union of sets, respectively. This operation allows us to construct for $\varphi_1\colon G \to \Aut(X_1)$ and $\varphi_2\colon G\to \Aut(X_2)$ a coproduct of these homomorphisms $\varphi_1\sqcup\varphi_2\colon G\to\Aut(X_1\coprod X_2)$, that is a direct sum of representations and an action on a disjoint union of sets, respectively.

Note that for any $Y\subseteq X\in \Cat$ there exists $Y^\perp\subseteq X$ such that $X=Y\coprod Y^\perp$. For a homomorphism $\varphi \colon G\to \Aut(X)$ and a $G$-invariant $Y\subseteq X$ we will call the corresponding restriction $\varphi|_X\colon G\to\Aut(X)$ a subhomomorphism. For any subhomomorphism $\varphi_1$ of $\varphi$ there exists another subhomomorphism $\varphi_1^\perp$ such that $\varphi=\varphi_1\sqcup\varphi_1^\perp$.

We will denote by $\Irr_{HS}(G)$ and $\Irr_P(G)$ sets of isomorphism classes of irreducible representations and transitive actions of a groups $G$, respectively. Let us denote the free abelian semigroup with the generators $\Irr(G)$ by $\rp(G)$.
$$\rp(G)=\bigoplus_{\tau \in \Irr(G)}\mathbb{Z}_{\ge 0}\tau .$$
Note that isomorphism classes of homomorphisms to $\Aut(X)\in\mathu$ correspond bijectively to elements of $\rp(G)$. 
For a homomorphism $\psi\colon G\to \Aut(X)$ we will denote the corresponding class $\psi^\#\in\rp(G)$.

We will denote by $\Lambda(G)$ the Grothendieck group of $\rp(G)$. 
$$\Lambda(G)=\bigoplus_{\tau \in \Irr(G)}\mathbb{Z}\tau .$$
We remark that $\Lambda_{P}(G)$ is the Burnside ring of a group $G$ and $\Lambda_{HS}(G)$ is the representation ring of $G$. 

We will denote by $1\in\Lambda(G)$ the class of one-dimensional trivial representation and the class of one point action, respectively. A tensor product operation in the respective categories gives a multiplication operation on $\Lambda(G)$ and $1$ is a multiplicative identity.

We will equip $\Lambda(G)$ with a norm $||\cdot||$ that is defined as following: 
$$\text{ for } \, v = \sum_{\tau \in\Irr(G)}\lambda_{\tau}\tau \quad 
||v|| :=  \sum_{\tau\in\Irr(G)}|\lambda_{\tau}|\cdot ||\tau||,$$
where $||\tau||$ is the dimension of the corresponding Hilbert space or the size of the corresponding set, respectively. For a homomorphism $\psi\colon G\to\Aut(X)$ we also have $||\psi^\#||=|X|$, where $|X|$ is the dimension or the size of $X$, respectively. Let us mention that for a finite group $H$ and $\tau\in\Irr(H)$ we always have $||\tau||\le |H|$. 

Standard theorems state that
each homomorphism $\psi\colon G\to \Aut(X)$ with $\psi^\#=\sum\limits_{\tau\in\Irr(G)}\lambda_\tau \tau$ has a canonical decomposition
    $$\psi = \coprod_{\tau \in\Irr(G)}\psi_\tau, 
    \quad \text{where} \quad 
    \psi_\tau^\#=\lambda_\tau \tau.$$

For $v,v'\in \Lambda(G)$ we define $\min(v,v')\in\Lambda(G)$ to be the element with each coordinate equal to the minimum of the coordinates of $v$ and $v'$. For homomorphisms $\psi_1\colon G\to \Aut(X)$ and $\psi_2\colon G \to \Aut(Y)$ the element $\min(\psi_1^\#,\psi_2^\#)$ corresponds to the maximal isomorphic subhomomorphisms of $\psi_1$ and $\psi_2$. 

For $v, v'\in \Lambda(G)$ we will say that $v\le v'$ if each coordinate of $v$ is not greater than the same coordinate of $v'$. 
Sometimes it will be convenient to consider ``$\min$'' and ``$\le$'' with respect to some other nonstandard generating set.

\subsection{Flexible stability}

For $\mathu_{HS}$ and $\mathu_P$ one can define a notion of (very) flexible stability that seems to be less restrictive, however, there is yet no evidence for this. We denote by $B(X)$ the space of all bounded operators on a Hilbert space $X$.

\begin{definition}\label{flexibleHS} 
    We will say that a group $G$ is very flexibly $HS$-stable if for any sequence of maps $\varphi_n \colon G \to \Aut(X_n)\in\mathu_{HS}$ satisfying 
    $$\forall g,h\in G \quad d(\varphi_n(g)\cdot\varphi_n(h),\varphi_n(gh))\xrightarrow[n\to\infty]{} 0 $$
    there exists a sequence of homomorphisms $\psi_n\colon G\to \Aut(Y_{n})$ for some $Y_n\supseteq X_n$ such that 
    $$\forall g\in G \quad ||\varphi_n(g)-p\psi_n(g)p||_{HS,X_n}\xrightarrow[n\to\infty]{} 0,$$
    where $p\in B(X_n)$ is a projection onto $Y_n$.

    If one can always find such $Y_n$ and $\psi_n$ with $|Y_n|/|X_n|\xrightarrow[n\to\infty]{} 1$, we will say that $G$ is flexibly $HS$-stable.
\end{definition}

\begin{definition}\label{flexibleP}   
    We will say that a group $G$ is very flexibly $P$-stable if for any sequence of maps $\varphi_n \colon G \to \Aut(X_n)\in\mathu_{HS}$ satisfying 
    $$\forall g,h\in G \quad d(\varphi_n(g)\cdot\varphi_n(h),\varphi_n(gh))\xrightarrow[n\to\infty]{} 0 $$ 
    there exists a sequence of homomorphisms $\psi_n\colon G\to \Aut(Y_{n})$ for some $Y_n\supseteq X_n$ such that
    $$\forall g\in G\quad  \frac{1}{|X_n|}\cdot \#\{x\in X_n: \varphi_n(g)(x)\ne \psi_n(g)(x)\} \xrightarrow[n\to\infty]{}0.$$

    If one can always find such $Y_n$ and $\psi_n$ with $|Y_n|/|X_n|\xrightarrow[n\to\infty]{} 1$, we will say that $G$ is flexibly $P$-stable. 
\end{definition}

These definitions of flexible stability look very different. We will give an equivalent definition that works in both cases. For this we will need to introduce a new operation, namely, a restriction to a subobject. 

For $\mathu_{P}$ the construction is straightforward. 
For a subset $X\subseteq Y$ and a permutation $a\in \Aut(Y)$ we define $a|_X\in \Aut(X)$ to be any permutation such that $a|_{X}(x)=a(x)$ for all $x\in X\cap a^{-1}X$. Note that 
\begin{equation}\label{eq1}
    \frac{1}{|X|}\cdot \#\{x\in X: a|_X(x)\ne a(x)\}\le \frac{|Y|-|X|}{|X|}.
\end{equation}

For $\mathu_{HS}$ we will first need to mention some preliminary results. The next lemma gives us the following useful inequality between the Hilbert-Schmidt norm and the operator norm (see, for example, Lemma 6.1 in \cite{gowers2017inverse}). 

\begin{lemma}\label{ineq_HS}
    For $a,b,c\in B(X)$ we have
    $$||abc||_{HS}\le ||a||_{op}\cdot||b||_{HS}\cdot||c||_{op}.$$
\end{lemma}

The following lemma gives a bound for a distance from some operator to a unitary one (see Lemma 2.2 in \cite{akhtiamov2020uniform}). We will use a little bit more data from the proof of this lemma
and will change its statement accordingly. 

\begin{lemma}\label{close_to_a_unitary}
    For the singular decomposition $b=u\sigma v$ of an operator $b\in B(X)$ we have
    $$||b-uv||_{HS}\le||b^*b-1_X||_{HS}.$$
\end{lemma}

For $\mathu_{HS}$ the construction of the restriction is the following. For a subspace $X\subseteq Y$ and an operator $a\in \Aut(Y)$ we define $a|_X:=uv$, where $pap=u\sigma v$ is a singular decomposition in $B(X)$ and $p\in B(Y)$ is a projection onto $X$. Lemma~\ref{close_to_a_unitary} and the fact that for $b\in B(X)\subseteq B(Y)$ we have 
$||b||_{HS,X}=||b||_{HS,Y}\sqrt{\frac{|Y|}{|X|}}$  
give us
$$||pap-a|_X||_{HS,X}\le ||(pap)^*pap-1_X||_{HS,X}\le 
(||(pap)^*pap-a^*a||_{HS,Y}+||1_Y-1_X||_{HS,Y})\sqrt{\frac{|Y|}{|X|}}.$$

Applying Lemma~\ref{ineq_HS}, Lemma~\ref{close_to_a_unitary} and the fact that $||p-1_Y||_{HS,Y}=\sqrt{\frac{|Y|-|X|}{|Y|}}$ we get
\begin{equation}\label{eq2}
    ||pap-a|_X||_{HS,X}\le 2||pap-a||_{HS,Y}\sqrt{\frac{|Y|}{|X|}}+\sqrt{\frac{|Y|-|X|}{|X|}}\le 5 \sqrt{\frac{|Y|-|X|}{|X|}}.
\end{equation}

For a group $G$ with a fixed generating set $S$ and for a homomorphism $\varphi\colon G\to \Aut(Y)$ we construct the restriction $\varphi|_{X}\colon F_S\to \Aut(X)$ on not necessarily $G$-invariant $X\supseteq Y$ by taking $\varphi|_{X}(s)=\varphi(s)|_X$.
Note that if $X\subseteq Y$ is $G$-invariant for $\varphi\colon G\to \Aut(Y)$ then $\varphi|_X$ coincides with the usual constructions of action/representation restriction. 

Inequalities (\ref{eq1}) and (\ref{eq2}) already allow us to reformulate the definitions of flexible stability in the following way. 

\begin{definition}
        We will say that a group $G$ is flexibly $\mathu$-stable if for any sequence of maps $\varphi_n \colon G \to \Aut(X_n)\in\mathu$ satisfying 
    $$\forall g,h\in G \quad d(\varphi_n(g)\cdot\varphi_n(h),\varphi_n(gh))\xrightarrow[n\to\infty]{} 0 $$ 
    there exists a sequence of homomorphisms $\psi_n\colon G\to \Aut(Y_{n})$ for some $Y_n\supseteq X_n$ such that
    $$\forall g\in G\quad  d(\varphi_n(g),\psi_n|_{X_n}(g)) \xrightarrow[n\to\infty]{} 0,\quad \frac{|Y_n|}{|X_n|}\xrightarrow[n\to\infty]{} 1.$$ 
\end{definition}

Now we will reformulate this definition in an equivalent way as we did with the definition of stability.  

\begin{definition}\label{def:stable_epi}
    Let us consider a group $\overline{G}$ generated by a finite set $S$. We will say that an epimorphism $\pi\colon\overline{G}\to G$ is flexibly $\mathu$-stable $(\mathu=\mathu_{P}$ or $\mathu=\mathu_{HS})$ if for any $\varepsilon>0$ there exist $\delta>0$ and a finite set $E\subseteq \ker \pi$ such that for any $\varphi\colon \overline{G}\to \Aut(X)\in\mathu$ with $d_E(\varphi, 1)\le\delta$ there exist $Y\supseteq X$ and $\psi\colon G\to \Aut(Y)$ satisfying $d_S(\varphi,\psi|_X)\le\varepsilon$ and $|Y|\le(1+\varepsilon)|X|$.
\end{definition}

One can repeat the proof of Lemma 3.1 from \cite{ioana2020stability} without any significant changes to get the following. 

\begin{prop}
    The group $G$ generated by a finite set $S$ is flexibly $\mathu$-stable if and only if the epimorphism $\pi\colon F_S\to G$ is flexibly $\mathu$-stable. 
\end{prop}

\subsection{Properties of the restriction}

To simplify the formulas we will use the constants $s_{P}=1$ and $s_{HS}=\frac{1}{2}$. 
We will write just $s$ when the statement holds both for $\mathu_P$ and $\mathu_{HS}$.
Also it will be convenient to write $f\preceq g$ when $f\le C\cdot g$ for some constant $C$ (that depends only on groups and their inclusions).    

For $a_1,a_2\in\Aut(X_1)$ and $b\in \Aut(X_2)$ we have  
$$d(a_1\sqcup b,a_2\sqcup b)=d(a_1,a_2)\cdot \left(\frac{|X_1|}{|X_1|+|X_2|}\right)^s.$$
As an immediate corollary we get the following lemma. 

\begin{lemma}\label{property_sums}
For two collections of elements $a_i,b_i\colon G\to \Aut(X_i)$, $i=1,\ldots, k$ we have 
$$d\left(\coprod_{i=1}^k a_i,\coprod_{i=1}^k b_i\right)\le \sum_{i=1}^k d(a_i,b_i)\cdot\left(\frac{|X_i|}{\sum_i |X_i|}\right)^s.$$
\end{lemma}
\begin{remark}
\label{rem:good_restriction}
    Directly from the restriction construction follows that for $a_i\in \Aut(X_i)$ and $Y_i\subseteq X_i$ we have 
    $$\left(\coprod_{i=1}^k a_i\middle)\right|_{\coprod Y_i}=\coprod_{i=1}^k a_i|_{Y_i}.$$
\end{remark}
The next lemma gives a very useful inequality.  
\begin{lemma}\label{lem_restr_useful}
    For $a\in\Aut(X)$ and $Y\subseteq X$ we have
    $$d(a,a|_Y\sqcup 1_{Y^\perp})\preceq \left(\frac{|X|-|Y|}{|X|}\right)^s.$$
\end{lemma} 
\begin{proof}
    The statement is clear for $\mathu_P$. For $\mathu_{HS}$ we see from Lemma~\ref{ineq_HS} that 
    $$||a-pap||_{HS,X}\le 2||p-1_X||=2\sqrt{\frac{|X|-|Y|}{|X|}},$$
    where $p$ is the projection onto $Y$. And the statement follows from this inequality and (\ref{eq2}).
\end{proof}

As an immediate corollaries of Lemma~\ref{lem_restr_useful} we get the following three lemmas.

\begin{lemma}\label{lem_double_restr}
    For $a\in \Aut(X)$ and $Z\subseteq Y\subseteq X$ we have 
    $$d((a|_Y)|_Z,a|_Z)\preceq \left(\frac{|X|-|Z|}{|Z|}\right)^s.$$
\end{lemma}


Note that if $Y$ is $a$-invariant then $(a|_Y)|_Z=a|_Z$. 

\begin{lemma}\label{property_restr}
For $Y\subseteq X$ and $a,b\in \Aut(X)$ we have
$$\left|d(a,b)-d(a|_Y,b|_Y)\cdot\left(\frac{|Y|}{|X|}\right)^s\right|\preceq\left(\frac{|X|-|Y|}{|X|}\right)^s.$$
\end{lemma}


\begin{lemma}\label{lem_almost_hom}
    For $Y\subseteq X$ and $a_i\in \Aut(X)$ we have
    $$d\left((a_1\cdot\ldots\cdot a_k)|_Y, a_1|_Y\cdot\ldots\cdot a_k|_Y\right)\preceq k \left(\frac{|X|-|Y|}{|X|}\right)^s.$$
\end{lemma}

These and similar restriction constructions are not new and were used to create so-called ``almost homomorphisms''. Let us explain what this means. 

We consider a group $G$ generated by a finite (not symmetric) set $S$ and the corresponding epimorphism $\pi\colon F_S\to G$. Given a homomorphism $\varphi\colon G\to\Aut(X)$ and $Y\subseteq X$, the restriction $\varphi|_Y\colon F_S\to \Aut(Y)$ is an ``almost homomorphism'' of a group $G$ in the following sense. For any $r\in\ker \pi$ we have
    $$d(\varphi|_Y(r),1_Y)\preceq |r|\left(\frac{|X|-|Y|}{|Y|}\right)^s.$$


Clearly, Lemmas 3-6 in this section also hold for homomorphisms and their restrictions, one needs only to replace $d(\cdot,\cdot)$ by $d_S(\cdot,\cdot)$.

\subsection{Amalgamated free products and HNN extensions}
Given three groups $G_1,G_2,H$ and inclusions $i_1\colon H \to G_1$, $i_2:H\to G_2$, the amalgamated free product of $G_1$ and $G_2$ over $H$ is the group defined by
$G_1*_HG_2:=G_1*G_2/\langle\langle E_0\rangle\rangle$, 
where $E_0=\{i_1(s)i_2(s)^{-1}: s\in S_H\}$ and $S_H$ is a generating set of $H$. 
Note that for any homomorphism $\varphi\colon G_1*G_2\to \Aut(X)\in\mathu$ we have 
$$d_{E_0}(\varphi, 1_X) = d_{S_H}(i_1^*(\varphi),i_2^*(\varphi)).$$ 

The special case of this construction is the double of a group $G$ over a subgroup $H$. In this case $G_1$ and $G_2$ are two copies of $G$, $i_1$ and $i_2$ are given by the same inclusion $i\colon H\to G$.

For a homomorphism $\varphi\colon G\to \Aut(X)$ and an element $t\in \Aut(X)$ we will denote by $\varphi^t\colon G\to \Aut(X)$ the homomorphism given by $\varphi^t(g):=t\varphi(g)t^{-1}$ for $g\in G$. 

For a group $G$ and two inclusions $i_1,i_2\colon H\to G$ the HNN extension of $G$ relative to $i_1,i_2$ is the group defined by 
$G*_{i_1,i_2}:= G*\mathbb{Z}/\langle\langle E_0\rangle\rangle$, 
where $\mathbb{Z}=\langle t\rangle$, $E_0=\{i_1(s)^{-1}i_2(s)^{t}: s\in S_H\}$ and $S_H$ is a generating set of $H$. Note that for any homomorphism $\varphi\colon G_1*\mathbb{Z}\to \Aut(X)\in\mathu$ we have 
$$d_{E_0}(\varphi, 1_X) = d_{S_H}(i_1^*(\varphi),i_2^*(\varphi)^{\varphi(t)}).$$

For a finite group $H$ we will always take $S_H=H$. If $G_j$ is generated by a set $S_j$ for $j=1,2$, we will take a generating set $S_1\sqcup S_2$ of $G_1*G_2$ and $G_1*_HG_2$. 

Both of these constructions can be seen as a quotient of a free product. The following proposition will be very useful.  

\begin{prop}\label{st_and_st_ep_give_st}
    If groups $H_1$ and $H_2$ are (flexibly) $\mathu$-stable, then $H_1*H_2$ is (flexibly) $\mathu$-stable.

    If a group $\overline{G}$ and an epimorphism $\pi\colon\overline{G}\to G$ are (flexibly) $\mathu$-stable, then $G$ is (flexibly) $\mathu$-stable. In other words, a composition of (flexibly) $\mathu$-stable epimorphisms is (flexibly) $\mathu$-stable. 
\end{prop}

\begin{proof}
    The first part for stability is clear. For flexible stability one can add a multiple of the trivial homomorphism to homomorphisms from $H_1$ and $H_2$ to get the same codomains. 

    The proof of the second part for stability is similar but simpler than for flexible stability. For this reason we will give the proof only for flexible stability. Let us fix a group $G_1$ generated by a finite set $S$ and two stable epimorphisms $\pi_1\colon G_1\to G_2$ and $\pi_2\colon G_2\to G_3$. Also take $\delta_j(\varepsilon_j)$, $E_j(\varepsilon_j)$, $Y_j(\varphi_j)$ and $\psi_j(\varphi_j)$ from Definition~\ref{def:stable_epi} for $\pi_j$ ($j=1,2$). 

    Let us fix some map $p\colon G_2\to G_1$ such that $\pi_1\circ p=Id_{G_2}$ and take $\delta =\delta_1$, $E=E_1\cup p(E_2)$.
    For a homomorphism $\varphi\colon G_1\to \Aut(X_1)$ satisfying $d_E(\varphi, 1)\le \delta_1$ there exist $X_2=Y_1(\varphi)\supseteq X_1$ and $\psi_1=\psi_1(\varphi)\colon G_2\to \Aut(X_2)$ such that $d_S(\varphi,\psi_1|_{X_1})\le \varepsilon_1$ and $|X_2|\le(1+\varepsilon_1)|X_1|$. Hence, by Lemma~\ref{property_restr} and Lemma~\ref{lem_almost_hom}
    \begin{equation}\label{eq1_fromprop}
    d_{E_2}(\psi_1,1_{X_2})\preceq \delta_1+\varepsilon_1\cdot \max\{|e| : e\in p(E_2)\}+\varepsilon_1^s.
    \end{equation}
    Since the right-hand side can be arbitrary small, we can apply stability of $\pi_2$ to $\psi_1$. This gives us $X_3=Y_2(\psi_1)\supseteq X_2$ and $\psi_2=\psi_2(\psi_1)\colon G_3\to \Aut(X_3)$ such that $d_S(\psi_1,\psi_2|_{X_2})\le \varepsilon_2$ and $|X_3|\le(1+\varepsilon_2)|X_2|$. 
    By Lemma~\ref{lem_double_restr} and Lemma~\ref{property_restr} we have 
    \begin{equation}\label{eq2_fromprop}
    d_S(\varphi,\psi_2|_{X_1})\preceq \varepsilon_1+\varepsilon_2\left(\frac{|X_2|}{|X_1|}\right)^s+\left(\frac{|X_3|-|X_1|}{|X_1|}\right)^s\le \varepsilon_1+\varepsilon_2 (1+\varepsilon_1)^s+(\varepsilon_1+\varepsilon_2+\varepsilon_1\varepsilon_2)^s.
    \end{equation}
    We see again that the right-hand side can be arbitrary small. This means that for each $\varepsilon>0$ we can fix small enough $\varepsilon_2>0$ using (\ref{eq2_fromprop}). This gives us $\delta_2$ and $E_2$. Using $\delta_2$, $E_2$ and (\ref{eq1_fromprop}) (and also (\ref{eq2_fromprop})) we take small enough $\varepsilon_1$. It gives us $\delta_1$ and $E_1$. We see that the stability condition for $\pi_2\circ\pi_1$ and $\varepsilon$ is satisfied for $\delta=\delta_1$ and $E=E_1\cup p(E_2)$.       
    
\end{proof}

In the proofs of stability we will always take the set $E_0$ as a set $E$ from the definition of a stable epimorphism.

For $\varphi\colon G_1*G_2\to \Aut(X)$ we will denote by $\varphi_1$ and $\varphi_2$ the restrictions of $\varphi$ on $G_1$ and $G_2$ correspondingly. For $\varphi\colon G*\mathbb{Z}\to\Aut(X)$ we will denote the restriction of $\varphi$ on $G$ by $\varphi_G$.

\section{Operator norm stability}
In \cite{eilers2020c} it was proven using results from \cite{blackadar1985shape}  that $C^*$-stability of groups is preserved under taking amalgamated free products and HNN extensions over finite groups.
Operator norm stability is weaker than $C^*$-stability (in fact, $C^*$-stability implies operator norm stability, see \cite{eilers2020c}). In this section we will show that the similar result holds for the operator norm stability as well. This gives more examples of finitely generated groups that are stable with respect to the operator norm.

Surprisingly, operator norm stability behaves much better in this setting than $HS$-stability and $P$-stability. This happens because of the following proposition that is probably well-known to experts. 

\begin{prop}\label{matr_prop}
Let $H$ be a finite group. Then for any $\varphi_1,\varphi_2\colon H\to \Aut(X)$ -- two finite-dimensional unitary representations of $H$ with $d_H(\varphi_1,\varphi_2)\le\delta<1$ there exists a unitary operator $u\in\Aut(X)$ such that
$$||u-1_X||_{op}\le 2\delta,\quad \varphi_2^u=\varphi_1.$$
\end{prop}

For the completeness of the exposition we will provide the proof. We will first need the following two lemmas. 
\begin{lemma}\label{posdef_lemma}
For a positive semi-definite operator $b\in B(X)$ and $u\in \Aut(X)$ if 
$u^b\in \Aut(X)$ then $u^b=u$.
\end{lemma}

\begin{proof}
We have:
$$1_X = (u^b)^*\cdot (u^b)=b^{-1}u^{-1}b^{2}ub^{-1} \quad \Longrightarrow \quad b^2u=u b^2.$$
So for any eigenvalue $\lambda$ of $b^2$ the corresponding eigenspace $X_\lambda$ is $u$-invariant. Since $b$ is positive semi-definite, $b$ and $b^2$ have the same eigenspaces. And we deduce that $bu=ub$. 
\end{proof}

\begin{lemma}[Lemma 4.5, \cite{burger2010ulam}]\label{fromUlam}
    Let $a\in B(X)$ with $||a-1_X||_{op}\le\delta<1$ and let $a=u\cdot b$ be the polar decomposition $(u\in\Aut(X)$, $b$ -- positive semi-definite$)$. Then $||u-1_X||_{op}\le 2 \delta$. 
\end{lemma}

We are ready to prove Proposition~\ref{matr_prop}.

\begin{proof}
We can consider the operator 
$$a = \frac{1}{|H|}\sum_{h\in H}\varphi_1(h)\varphi_2(h)^{-1}.$$
Let us note that 
$$||a-1_X||_{op}\le \frac{1}{|H|}\sum_{h\in H}||\varphi_1(h)\varphi_2(h)^{-1}-1_X||_{op}=\frac{1}{|H|}\sum_{h\in H}||\varphi_1(h)-\varphi_2(h)||_{op}\le d_H(\varphi_1,\varphi_2)\le \delta $$
and
$$\forall h\in H \quad   \varphi_1(h)a\varphi_2(h)^{-1}=a\quad  \Longleftrightarrow \quad \varphi_1 =\varphi_2^a. $$
Let us consider the polar decomposition $a=ub$ of $a$. 
By Lemma~\ref{fromUlam} we have 
$$||u-1_X||_{op}\le 2\delta.$$ 
We can apply Lemma~\ref{posdef_lemma} to $b$ and $\varphi_2(h)$ since $\varphi_2(h)^b=\varphi_1(h)^{u^{-1}}\in\Aut(X)$. 
We get $\varphi_2^u=\varphi_1$, this finishes the proof. 
\end{proof}

\begin{theorem}\label{operator_amalgam}
Let $H$ be a finite group, let $G_1$ and $G_2$ be operator norm stable groups and let ${i_j:H\to G_j}$ for $j=1,2$ be some inclusions. Then $G_1*_H G_2$ is operator norm stable.  
\end{theorem}
\begin{proof}
 
For any $\varphi\colon G_1*G_2\to\Aut(X)$ with $d_H(i_1^*(\varphi),i_2^*(\varphi))\le \delta < 1$ by Proposition~\ref{matr_prop} there exists $u\in \Aut(X)$ such that $i_1^*(\varphi)=i_2^*(\varphi)^u$ and $||u-1_X||_{op}\le 2\delta$. 
We can construct the representation $\psi\colon G_1* G_2\to \Aut(X)$ by
$$\psi_{G_1}:=\varphi_{G_1}, \quad 
\psi_{G_2}=\varphi_{G_2}^u.$$
Note that $i_1^*(\psi)=i_2^*(\psi)$ and 
$$ d_{S_1\sqcup S_2}(\varphi, \psi)\le 2 d(u, 1_X)\le 4\delta.$$
Since $4\delta$ can be arbitrary small, we finished the proof.
\end{proof}

\begin{theorem}\label{operator_HNN}
Let $H$ be a finite group, let $G$ be an operator norm stable group, let $i_1,i_2\colon H \to G$ be two inclusions. Then $G*_{i_1,i_2}$ is operator norm stable. 
\end{theorem}
\begin{proof}
Given any $\varphi\colon G*\mathbb{Z}\to\Aut(X)$ with $d_H(i_1^*(\varphi),i_2^*(\varphi)^{\varphi(t)})\le \delta < 1$,
we can apply Proposition~\ref{matr_prop} and get $u\in Aut(X)$ such that $d(u,1_X)\le 2\delta$ and $i_1^*(\varphi)=i_2^*(\varphi)^{u\varphi(t)}$.
We will define $\psi\colon G*\mathbb{Z}\to \Aut(X)$ by 
$$\psi_G:=\varphi_G, \quad \psi(t):=u\varphi(t).$$
By construction $d_{S\sqcup\{t\}}(\varphi,\psi) \le d(u,1_X)\le 2 \delta$ and $i_1^*(\psi)=i_2^*(\psi)^{\psi(t)}$. This finishes the proof.
\end{proof}

\begin{exmp}
12 out of 17 wallpaper groups are operator norm stable (see Figure 1 in \cite{eilers2020c}, where operator norm stability is called matricial stability). 11 of them have torsion and can be used to produce non-amenable finitely presented operator norm stable groups (that are not virtually free).     
\end{exmp}

\section{Stability of group doubles}
We now move to proving $\mathu$-stability of amalgamated free products and HNN extensions for $\mathu=\mathu_{HS}$ and $\mathu=\mathu_P$.

\subsection{``The conjugator argument''}

The main ingredient of the proofs is the following fact: if two conjugated actions/representations of a finite group are close, then they are conjugated by an element that is close to identity. 
We have already seen how the similar fact works in the case of operator norm stability. The only difference for $\mathu_P$ and $\mathu_{HS}$ is that we need to assume that homomorphisms are conjugated.   
We will call this fact for simplicity ``the conjugator argument''. 

\begin{prop}\label{prop_conj}
    Let $H$ be a finite group and let $\varphi_1,\varphi_2\colon H\to \Aut(X)\in \mathu$ be two homomorphisms. If $\varphi_1^\#=\varphi_2^\#$, then there exists $t\in\Aut(X)$ such that 
    $$d(t, 1_X)\preceq d_H(\varphi_1,\varphi_2), \quad \varphi_2^t=\varphi_1.$$
\end{prop}

For $\mathu_{HS}$ this proposition follows from Theorem 7.3 in \cite{gowers2017inverse} (see Corollary 2 in \cite{ger2021}).

\begin{proof}[Proof of Proposition~\ref{prop_conj} for $\mathu_P$]
 Let us define a new set $X_0:=\{x\in X:\forall h\in H \quad  \varphi_1(h)x=\varphi_2(h)x\}$. For the sets $X_h= \{x\in X : \varphi_1(h)x=\varphi_2(h)x\}$ it is easy to see that $|X\setminus X_h|=|X|\cdot d(\varphi_1(h),\varphi_2(h))$ and $X\setminus X_0 = \cup_{h\in H} X\setminus X_h$, hence we get $|X\setminus X_0|\le |H|\cdot|X|\cdot d_H(\varphi_1,\varphi_2)$. 
 For any $x\in X_0$ and any $h,h'\in H$ we have
 $$\varphi_1(h')\varphi_1(h)x = \varphi_1(h'h)x=\varphi_2(h'h)x=\varphi_2(h')\varphi_2(h)x=\varphi_2(h')\varphi_1(h)x.$$
 Thus, $X_0$ is $\varphi_1(H)$-invariant, i.e. it is a union of  $\varphi_1(H)$-orbits on which the actions $\varphi_1$ and $\varphi_2$ coincide.   
 Since $\varphi_1|_{X_0}^\#=\varphi_2|_{X_0}^\#$ and $\varphi_1^\#=\varphi_2^\#$, we conclude that $\varphi_1|_{X\setminus X_0}^\#=\varphi_2|_{X\setminus X_0}^\#$. Therefore, there exists $t_0\in \Aut(X\setminus X_0)$ such that $(\varphi_2|_{X\setminus X_0})^{t_0} =\varphi_1|_{X\setminus X_0}$. 
 Now we can take $t=1_{X_0}\sqcup t_0$. 
 By construction we have
 $$d(t,1_X)\le|X_0|\preceq d_H(\varphi_1,\varphi_2) \quad \text{and}\quad \varphi_2^t=\varphi_1.$$
\end{proof}

Using Proposition~\ref{prop_conj} we can already establish stability of some HNN extensions. In this special case the proof will be exactly the same as the proof of Theorem~\ref{operator_HNN}.

\begin{theorem}\label{double_HNN}
    Let $H$ be a finite group, let $G$ be a $\mathu$-stable group and let $i\colon H\to G$ be an injective homomorphism. Then HNN extension $G*_{i,i}$ is $\mathu$-stable. 
\end{theorem}

Let us also mention an interesting property that connects $||\cdot||$ on $\Lambda(H)$ and the distance between homomorphisms. 

\begin{lemma}\label{distances_connection}
    Let $H$ be a finite group. Then for $\varphi_1,\varphi_2\colon H\to \Aut(X)$ we have
    $$||\varphi_1^\#-\varphi_2^\#||\preceq d_H(\varphi_1,\varphi_2)^\frac{1}{s}|X|.$$

    The ``converse'' is also true. For any $\varphi_1\colon H\to\Aut(X)$ and any $\xi\in\rp(H)$ satisfying $||\xi||=|X|$ there exists $\varphi_2\colon H\to\Aut(X)$ such that 
    $$d_H(\varphi_1,\varphi_2)|X|^s\preceq ||\varphi_1^\#-\xi||^s\quad\text{and}\quad \varphi_2^\#=\xi.$$
\end{lemma}
\begin{proof}
    The first part of the statement is the special case of Proposition 4.3 in \cite{lazarovich2021virtually} for $\mathu_P$ and is the part of Theorem 7.3 in \cite{gowers2017inverse} for $\mathu_{HS}$. 
    The second part follows from Lemma 6.2 in \cite{lazarovich2021virtually} for $\mathu_P$ and from Lemma 4 in \cite{ger2021} for $\mathu_{HS}$.
\end{proof}

\subsection{Sufficient conditions for flexible stability of amalgamated free products}
From Proposition~\ref{prop_conj} we see that in general for $\mathu_P$ and $\mathu_{HS}$ there is an additional ``discrete'' problem compared to $\mathu_{op}$ -- one needs to find an appropriate element in $\rp$ before constructing the homomorphism of a group. Let us show how the precise statement looks for flexible stability of an amalgamated free product.  

\begin{theorem}\label{main_flex}
    Assume that for any $\varepsilon>0$ there exists $\delta>0$ such that
    for any $\xi_j \in \rp(G_j)$ $(j=1,2)$ satisfying 
    $$||\xi_1||=||\xi_2||, \quad  ||i_1^*(\xi_1)-i_2^*(\xi_2)||\le \delta ||\xi_1||$$
    there exist $\xi_j',\eta_j\in \rp(G_j)$ such that
    \begin{equation}\label{conditions_flex}
        \xi_j'\le \xi_j, \quad i_1^*(\xi_1'+\eta_1)=i_2^*(\xi_2'+\eta_2), \quad ||\eta_j||, ||\xi_j-\xi_j'|| \le \varepsilon ||\xi_1||.
    \end{equation}
    Then the epimorphism $\pi\colon G_1*G_2 \to G_1*_H G_2$ is flexibly $\mathu$-stable.  
\end{theorem}
    Informally speaking, this means that we want to ``cut’’ some small parts of homomorphisms of $G_j$ and then add something small to get some new homomorphisms of $G_j$ with close and isomorphic restrictions to $H$. 
\begin{proof}
    Let us fix any $\varepsilon>0$ and $\delta=\delta(\varepsilon)>0$ as above. 
    Any homomorphism  $\varphi\colon G_1*G_2\to \Aut(X)\in\mathu$ satisfying $d_{H}(i_{1}^*(\varphi),i_2^*(\varphi))\le \delta^s$ by Lemma~\ref{distances_connection} gives us $\xi_j:=\varphi_{j}^\#\in \rp(G_j)$ that satisfy 
    $$||i_1^*(\xi_1)-i_2^*(\xi_2)|| \preceq \delta ||\xi_1||, \quad ||\xi_1||=||\xi_2||.$$
    By our assumption, there exist $\xi_j', \eta_j \in \rp(G_j)$ such that conditions (\ref{conditions_flex}) hold. 
    We can always add a multiple of $1\in\rp(G_j)$ both to $\eta_1$ and to $\eta_2$, so we can assume that $||\xi_j'||+||\eta_j||\ge ||\xi_j||$. This changes $||\eta_j||$ by at most $||\xi_j-\xi_j'||$, so conditions (\ref{conditions_flex}) still hold for $2\varepsilon$ instead of $\varepsilon$. 
    Let us take $\varphi(G_j)$-invariant subobjects $X_j\subseteq X$ such that $\varphi_{j}|_{X_j}^\# =\xi_j'$ (they exist since $\xi_j'\le\xi_j$).
    Let us also take some homomorphisms $\alpha_j\colon G_j\to \Aut(Y_j)\in\mathu$ such that $\alpha_j^\#=\eta_j$, where $Y_j\in\Cat$ is some object with $|Y_j|=||\eta_j||$. 

    We can identify $X_1\coprod Y_1$ and $X_2\coprod Y_2$ with some $Z\supseteq X$ of a proper size. 
    By construction $Z$ is equipped with a homomorphism $\beta\colon G_1*G_2\to \Aut(Z)$, $\beta_{j}=\left(\varphi_{j}|_{X_j}\right)\sqcup \alpha_j$.   
    Note that $X_j$ is $G_j$-invariant and $\varphi_j|_{X_j}=\beta_j|_{X_j}$, hence $\varphi_{j}|_{X_1\cap X_2}=\beta_{j}|_{X_1\cap X_2}$. 
    Also we have $d_H(i_1^*(\varphi),i_2^*(\varphi))\le \delta^s$.
    So by Lemma~\ref{property_restr} we get  
    $$d_H(i_1^*(\beta)|_{X_1\cap X_2},i_2^*(\beta)|_{X_1\cap X_2})\preceq
    \delta^s\cdot\left(\frac{|X|}{|X_1\cap X_2|}\right)^s+\left(\frac{|X|-|X_1\cap X_2|}{|X_1\cap X_2|}\right)^s.$$
    By construction $|Z|-|X_1\cap X_2|\preceq \varepsilon |X|\le\varepsilon |Z|$. Applying again Lemma~\ref{property_restr} we get
    $$d_H(i_1^*(\beta),i_2^*(\beta))\preceq\delta^s\cdot\left(\frac{|X|}{|Z|}\right)^s+\left(\frac{|X|-|X_1\cap X_2|}{|Z|}\right)^s+\left(\frac{|Z|-|X_1\cap X_2|}{|Z|}\right)^s\preceq \delta^s+\varepsilon^s.$$

    We are ready to apply the conjugator argument. By Lemma~\ref{prop_conj} there exists $t\in \Aut(Z)$ such that 
    $$d(t,1)\preceq \delta^s+\varepsilon^s,\quad (i_2^*(\beta))^t=i_1^*(\beta).$$
    We define $\psi\colon G_1*G_2\to \Aut(Z)$ by $\psi_1=\beta_1$ and $\psi_2=\beta_2^t$. By construction $i_1^*(\psi)=i_2^*(\psi)$ and $|Z|-|X|\preceq \varepsilon |X|$. 
    We have by Lemma~\ref{property_restr}
    $$d_{S_1\sqcup S_2}(\beta|_X, \psi|_X)\preceq (\delta^s+\varepsilon^s)\cdot \left(\frac{|Z|}{|X|}\right)^s+\left(\frac{|Z|-|X|}{|X|}\right)^s\preceq \delta^s+\varepsilon^s,$$
    $$d_{S_j}(\varphi_j, \beta_j|_{X})\preceq \left(\frac{|X|-|X_j|}{|X|}\right)^s\preceq \varepsilon^s.$$
    Hence, 
    $$d_{S_1\sqcup S_2}(\varphi, \psi|_X)\preceq d_{S_1\sqcup S_2}(\varphi, \beta|_X)+ d_{S_1\sqcup S_2}(\beta|_X,\psi|_X)\preceq \delta^s+\varepsilon^s.$$
    Since $\delta(\varepsilon)^s+\varepsilon^s$ and $\varepsilon$ can be made arbitrary small for small enough $\varepsilon$, we finished the proof.    
\end{proof}

If one takes $\xi_j'=\xi_j$, then the statement of Theorem~\ref{main_flex} can be reformulates in the following way.  

\begin{cor}\label{cor_flex}
Assume that for any $\varepsilon>0$ there exists $\delta>0$ such that for any $\xi_j\in i_j^*(\rp(G_j))$ $(j=1,2)$ with 
$||\xi_1-\xi_2||\le\delta||\xi_1||$ and $||\xi_1||=||\xi_2||$ there exist $\eta_j\in i_j^*(\rp(G_j))$ such that 
$$\xi_1+\eta_1=\xi_2+\eta_2, \quad ||\eta_j||\le \varepsilon ||\xi_j||.$$
Then the epimorphism $\pi\colon G_1*G_2 \to G_1*_H G_2$ is flexibly $\mathu$-stable.  
\end{cor}

This corollary is especially interesting, since instead of $\rp(G)$ that can be infinitely generated it uses $i_j^*(\rp(G_j))$ that lies inside a finitely generated free abelian group $\Lambda(H)$. 
To apply Corollary~\ref{cor_flex} we need to do some preliminary work.

\subsection{Stability of group doubles}
We will start with some definitions and technical lemmas.
\begin{definition}
We will call a semigroup generated by a subset $A\subset \mathbb{Z}^N$ and $0$ a cone generated by $A$ $($a subcone of $\mathbb{Z}^N)$.  
\end{definition}

By $\rcone$ we denote the subsemigroup of $\mathbb{R}^N$ consisting of all linear combinations of vectors of $\cone$ with non-negative coefficients. 
We will denote by $\latt$ the abelian group generated by $K$.
In this subsection we will write $f\preceq g$ when $f\le C\cdot g$ for some constant $C$ that depends only on cones.

The following lemma, informally speaking, means that cones can be not maximally dense near borders but they are maximally dense sufficiently far from them.  

\begin{lemma}\label{dense}
For any finitely generated cone $\cone$ there exists a vector $w_0\in\cone$ such that 
$$\latt\cap(w_0+\rcone)\subseteq \cone.$$  
\end{lemma}

\begin{proof}
Let us take a finite generating set $\{v_i\}, i=1\dots k$ for $\cone$. We will consider $w_0$ of the form $w_0=L\sum_{i} v_i$, where $L$ is a large enough number.
For any vector $w\in \latt\cap(w_0+\rcone)$ we know that
$$w = \sum_{i} k_i\cdot v_i, \quad k_i\in\mathbb{Z},$$
$$ w = \sum_{i} \nu_i\cdot v_i, \quad \nu_i\ge L.$$
The set of all integer relations between $v_i$ forms a finitely generated abelian group. Let $\lambda^j$ be its generating set, each $\lambda^j$ is a vector in $\mathbb{Z}^k$ with coordinates $\lambda^j_i\in\mathbb{Z}$. We have 
$$\sum_i(k_i-\nu_i)v_i=0 .$$
Hence, 
$$\nu_i-k_i=\sum_j \alpha_j \lambda_i^j .$$
For $L= \sum_{i,j}|\lambda_i^j|$ we have 
$$\forall i \quad \sum_j(\alpha_j-[\alpha_j])\lambda_i^j\le L,$$  
where $[x]$ is an integer part of $x$. Now we can take
 $$\nu_i'=k_i+\sum_j[\alpha_j]\lambda_i^j=\nu_i-\sum_j(\alpha_j-[\alpha_j])\lambda_i^j.$$ 
From the first part of the equality we can see that $\nu_i'\in \mathbb{Z}$ and from the second part we can see that $\nu_i'\ge 0$. One can see that $w=\sum_i\nu_i' v_i$, hence $w\in\cone$. 
\end{proof}

\begin{prop}\label{double_prop}
For any cone $K\subseteq\mathbb{Z}_{\ge 0}^N$ and any $\xi_1, \xi_2\in \cone$ there exist $\eta_1, \eta_2\in\cone $ such that 
$$\xi_1+\eta_1=\xi_2+\eta_2,\quad ||\eta_1||,||\eta_2||\preceq ||\xi_1-\xi_2||.$$
\end{prop}
\begin{proof}
    Since all norms on $\mathbb{R}^N$ are equivalent, it is sufficient to prove this statement for the Euclidean norm. Let us define the sets $A_n:=\{\xi\in \cone: |\xi|\le n\}$ and denote $\cone_n:=\mathbb{Z}_{\ge 0}A_n$. Note that $\cone=\bigcup_n \cone_n$ and for $n'>n$, $\cone_n\subseteq \cone_{n'}$. Each element of $\mathbb{Z}\cone$ is in $\mathbb{Z}\cone_n$ for $n$ large enough, so there exists $n_0$ such that $\mathbb{Z}\cone_{n_0}=\mathbb{Z}\cone$. Let us take some $u\in K_{n_0}\cap \Int(\rcone_{n_0})$ (here the interior is taken inside $\rlatt$) and $\gamma>0$ such that $\gcone_{\gamma}(u)\subseteq \rcone_{n_0}$, where $\gcone_{\gamma}(u):=\{\xi\in\rlatt: \angle(\xi,u)\le\gamma\}$. 

    For any $\xi_1,\xi_2\in\cone$ if $\xi_1=\xi_2$ there is nothing to prove, so we can assume that $\xi_1\ne\xi_2$ and $1\le |\xi_1-\xi_2|$.
    
    One can consider two-dimensional affine space $\xi_1+span(\xi_1-\xi_2, u)$ (see Figure~\ref{fig:only}) and observe that not only $\xi_1+\gcone_{\gamma}(u)\cap\xi_2+\gcone_{\gamma}(u)\ne\emptyset$ but there exist $\eta_1, \eta_2 \in \gcone_{\gamma}(u)$ such that 
    $$\xi_1+\eta_1=\xi_2+\eta_2,\quad |\eta_1|,|\eta_2|\le \frac{|\xi_1-\xi_2|}{\sin\gamma}.$$
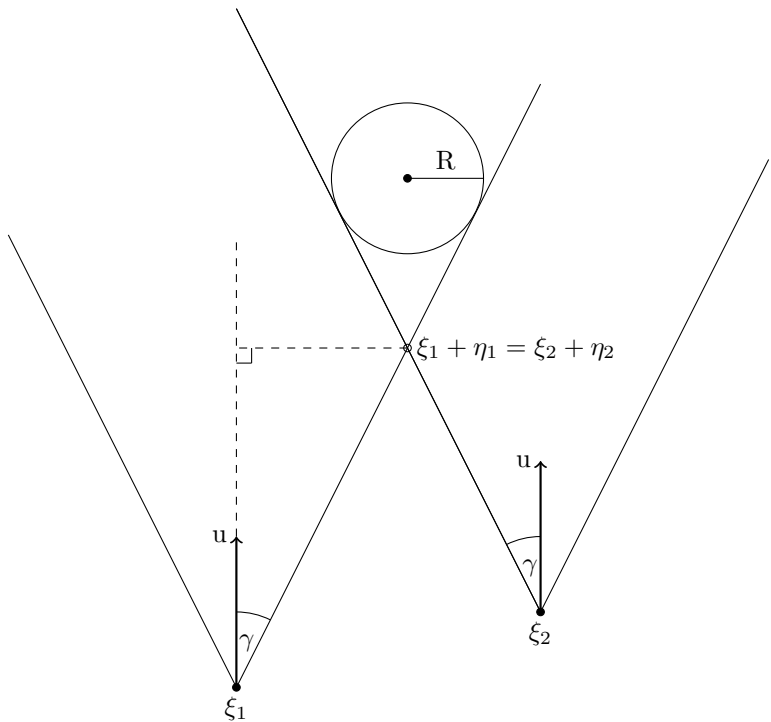
\begin{figure}[h]
    \centering
    \begin{tikzpicture}
    \draw (1,1) -- (-2,7); 
    \draw (5,2) -- (8,8); 
    \draw (5,2) -- (1,10); 
    \draw (3.25,7.75) circle (1cm);
    \draw[thick, ->] (5,2) -- (5,4);
    \draw[thick, ->] (1,1) -- (1,3);
    \draw[dashed] (1,3) -- (1,7);
    \draw[dashed] (3.25,5.5) -- (1,5.5);
    \draw (1.2,5.5) -- (1.2,5.3);
    \draw (1,5.3) -- (1.2,5.3);
    \draw (3.25, 7.75) -- (4.25, 7.75);
    \draw (3.75, 7.75) node[anchor=south] {R};
    \draw[fill] (1,1) circle (0.05) node[anchor=north] {$\xi_1$};
    \draw[fill] (5,2) circle (0.05) node[anchor=north] {$\xi_2$};
    \draw[fill] (3.25,7.75) circle (0.05);
    \draw (5,9) coordinate (A) -- (1,1) coordinate (B) -- (1,3) coordinate (C) node[anchor=east] {u} pic ["$\gamma$",draw, angle radius= 1cm] {angle = A--B--C};
    \draw (1,10) coordinate (A1) -- (5,2) coordinate (B1) -- (5,4) coordinate (C1) node[anchor=east] {u} pic ["$\gamma$",draw, angle radius= 1cm] {angle = C1--B1--A1};
    \draw (3.25,5.5) circle (0.05) node[anchor=west] {$\xi_1+\eta_1=\xi_2+\eta_2$};
    \end{tikzpicture}
    \caption{$\xi_1+span(\xi_1-\xi_2, u)$}
    \label{fig:only}
\end{figure}
    
    Since $\latt$ is cocompact in $\rlatt$, there exists $R$ such that any ball in $\rlatt$ of radius $R$ contains at least one point of $\latt$. Also $\gcone_{\gamma}(u)$ contains some ball of radius $R$ (taken in $\rlatt$), so we can change $\eta_1$ and $\eta_2$ by adding the same vector of at most constant length from $\gcone_{\gamma}(u)$ so that $\eta_1,\eta_2\in\latt=\latt_{n_0}$ (they are still in $\gcone_{\gamma}(u)\subseteq \rcone_{n_0}$).  

    On the last step we change $\eta_1$ and $\eta_2$ again by adding $w_0=w_0(K_{n_0})$ from Lemma~\ref{dense} so that $\eta_1,\eta_2\in \cone_{n_0}\subseteq\cone$. Other conditions are satisfied by construction, since $1\le |\xi_1-\xi_2|$.
\end{proof}

\begin{theorem}\label{double_amalgam}
If $\rk(\Lambda(H))<\infty$ and a group $G$ is flexibly $\mathu$-stable, then the group double $G*_HG$ is flexibly $\mathu$-stable.
\end{theorem}
\begin{proof}
We consider the cone $\cone=i^*(\rp(G))\subseteq \rp(H)$. 
Proposition~\ref{double_prop} for the cone $\cone$ proves exactly the conditions from Corollary~\ref{cor_flex} for the Euclidean norm. But all norms on $\Lambda(H)$ are equivalent, so we finished the proof. 
\end{proof}

Note that the condition $\rk(\Lambda(H))<\infty$ holds for finite groups. For $\mathu_{HS}$ it is equivalent to having finitely many conjugacy classes and for $\mathu_P$ it is equivalent to having finitely many finite index subgroups. 
Each of these conditions holds for some infinite groups as well. 

\begin{remark}\label{remark_1}
We note that in general there is no hope that a group double $G*_H G$ of a stable group $G$ over an infinite stable group $H$ will be stable (or in some cases even very flexibly stable). For example, $F_2\times F_2$ can be written as a group double $(F_2\times \mathbb{Z})*_{F_2} (F_2\times \mathbb{Z})$ and known to be not $HS$-stable $($see \cite{ioana2021almost}$)$, but $F_2\times \mathbb{Z}$ is $HS$-stable $($see \cite{ioana2019ii1}$)$. 
For $\mathu_P$ one can consider $F_2\times\mathbb{Z}=\mathbb{Z}^2*_\mathbb{Z}\mathbb{Z}^2$, that is not even very flexibly $P$-stable by \cite{ioana2020stability}, but $\mathbb{Z}^2$ is $P$-stable by \cite{arzhantseva2015almost}.
\end{remark}

\begin{remark}
We also note that a group double of a stable group over a finite index subgroup is also not necessarily stable (or even flexibly stable).
Let $G$ be a virtually free group (non-abelian), let $H$ be a subgroup of $G$ with $3\le[G:H]<\infty$, then the group double of $G$ over $H$ is not very flexibly stable in permutations.  
Indeed, by Corollary 1.6 from \cite{benakli2001note} every such double is virtually $F_{r_1}\times F_{r_2}$ for $r_1,r_2\ge 2$ which is not very flexibly stable by \cite{ioana2020stability}. Lemma 3.3 from \cite{ioana2020stability} tells that very flexible $P$-stability of a group implies very flexible stability of all its subgroups of finite index. Hence, the double $G*_HG$ is not very flexibly $P$-stable.
\end{remark}

\begin{exmp}
    One can take his favorite stable group with a finite subgroup and consider its double. For example, let $G$ be the Grigorchuk's group $($it is $P$-stable by \cite{zheng2019rigid}$)$, then the double $G*_{\mathbb{Z}_2}G$ is flexibly $P$-stable.
\end{exmp}

\section{Stability and properties of $i^*$}

\subsection{Properties of $i^*$}

For $i\colon H\to G$ there is a natural homomorphism $i^*\colon \rp(G) \to \rp(H)$. When $H$ is finite, properties of $i^*$ play an important role in the proofs of flexible stability as we will see later. We start with the following proposition.

\begin{prop}\label{about_A0}
  Let $H$ be a finite group and let $i\colon H \to G$ be an inclusion. Then for $i^*\colon \rp(G)\to \rp(H)$ and $\tau\in \Irr(G)$ the following properties hold:  
\begin{enumerate}
    \item If $H$ is central in $G$, then $i^*(\tau)=k\sigma$ for some $k\in\mathbb{Z}_{>0}$, $\sigma\in\Irr(H)$. 
    
    \item If $H$ is normal in $G$, then $i^*(\tau)=k\sum\limits_{\sigma' \in G\cdot \sigma}\sigma'$ for some $k\in\mathbb{Z}_{>0}$, $\sigma\in\Irr(H)$.  

    \item If $[G:H]<\infty$ and $i^*(\tau)=\sum_{j}\sigma_j$, where $\sigma_j\in\Irr(H)$, then $ ||\tau||\le [G:H]\cdot ||\sigma_j||$ for all $j$. In particular, the number of summands in the decomposition of $i^*(\tau)$ above is not greater than $[G:H]$.  
    
    \item Moreover, if $[G:H]<\infty$ and $i^*(\tau)=\sum_{j
    }\sigma_j$ for some $\sigma_j\in\Irr(H)$, then
    $$\frac{ ||\sigma_j||}{||\tau||}=\frac{k_j}{n}$$
    for some $k_j,n \in \mathbb{Z}_{>0}$, $n\le [G:H]^{[G:H]}$.
    
    \item If $H$ is almost normal in $G$, then there exists a finite set $\aleph=\aleph(i)$ such that $i^*(\tau)=k a$ for some $k\in\mathbb{Z}_{>0}$ and $a\in \aleph$. 
\end{enumerate}
\end{prop}
Let us note that (1) also holds when $G$ is generated by $H$ and $C_G(H)$. 
\begin{proof}
    Let us a choose a homomorphism $\varphi\colon G\to \Aut(X)$ with $\varphi^\#=\tau\in \Irr(G)$. For $i^*(\varphi)$ there is a unique decomposition $X=\coprod_\sigma X_\sigma$, $\sigma \in \Irr(H)$. For $g\in C_G(H)$ we have $\varphi(g)X_\sigma = X_\sigma$.
    If $H$ is central in $G$, each $X_\sigma$ is $G$-invariant, hence $X=X_{\sigma}$ for some $\sigma \in \Irr(H)$ and we proved (1).

    The normalizer $N_G(H)$ acts on $H$ by conjugation, this gives an action of $N_G(H)$ on $\Irr(H)$. For $g\in N_G(H)$ we have $\varphi(g)X_\sigma \subseteq X_{g\cdot \sigma}$. Hence, $|X_\sigma|=|X_{g\cdot \sigma}|$ and $||\sigma ||= ||g\cdot\sigma||$. 
    If $H$ is normal, then $\coprod_{\sigma' \in G\cdot \sigma}X_{\sigma'}$ is $G$-invariant for any $\sigma\in \Irr(H)$. Thus, 
    $X=\coprod_{\sigma' \in G\cdot \sigma}X_{\sigma'}$ for some $\sigma \in \Irr(H)$ and we have $i^*(\tau)=k\sum_{\sigma' \in G\cdot \sigma}\sigma'$ for $k=\frac{|X_\sigma|}{||\sigma||}\in\mathbb{Z}_{>0}$. This proves (2). 
    
    The decomposition $i^*(\tau)=\sum_{j}\sigma_j$ corresponds to some decomposition $X=\coprod_{j} X_j$ for $H$-invariant $X_j$. For each $j$ denote by $Y_j\subseteq X$ the minimal subobject containing all $\varphi(g)X_j$ for $g\in \{G/H\}$, where $\{G/H\}\subseteq G$ is a set of representatives for the left cosets $G/H$. Note that $Y_j$ are $G$-invariant.  Hence, $X=Y_j$ for some $j$ and we have $||\tau||=|Y_j|\le [G:H]\cdot |X_j|={[G:H]\cdot ||\sigma_j||}$.
    
    If $[G:H]<\infty$, let us denote by $H_0\trianglelefteq G$ the intersection of all subgroups conjugated to $H$. Note that there are at most $[G:H]$ such subgroups and each has the same index in $G$ as $H$, so $[G:H_0]\le [G:H]^{[G:H]}$. We can apply (2) to the restriction from $G$ to $H_0$ and get the following images of $\tau$ in $\rp(H)$ and $\rp(H_0)$
    
    $$\tau \quad \longmapsto \quad  \sum_{j=1}^m \sigma_j \quad \longmapsto \quad k\sum_{\sigma' \in G\cdot \sigma}\sigma'$$
    for some $\sigma_j\in \Irr(H)$, $\sigma\in\Irr(H_0)$ and $m,k\in \mathbb{Z}_{>0}$. We see that for each $j$ there exists some $k_j\in\mathbb{Z}_{>0}$ such that $||\sigma_j||=k_j\cdot||\sigma||$. Let us denote $n:=k\cdot |G\cdot\sigma|$.  We have $||\tau||=n\cdot||\sigma||$ and by (3) also $n\le[G:H_0]\le[G:H]^{[G:H]}$.
    This finishes the proof of (4). 

    For an almost normal finite subgroup $H\le G$ we can consider a subgroup $H_1$ such that ${H\trianglelefteq H_1\le G}$ and $[G:H_1]<\infty$. The sequence $H\trianglelefteq H_1\le G$ corresponds by (2) to the following sequence of images
    $$\tau \quad \longmapsto\quad  \sum_{j=1}^{m} \sigma_j \quad \longmapsto \quad \sum_{j=1}^m \left(k_j \cdot \sum_{\theta' \in H_1\cdot \theta_j}\theta'\right)=
    \sum_{\sigma\in \Irr(H)}\lambda_\sigma \cdot \sigma$$
    for some $\sigma_j\in \Irr(H_1)$, $\theta_j\in \Irr(H)$ and $k_j,m,\lambda_\sigma \in \mathbb{Z}_{\ge 0}$. Note that $m\le [G:H_1]$ by (3). 
    
    Let us prove that $\frac{\lambda_\sigma}{||\tau||}$ can take only finitely many values. 
    Note that $k_j\cdot |H_1\cdot\theta_j|\cdot ||\theta_j||=||\sigma_j||$. We can write
    $$\frac{\lambda_\sigma}{||\tau||}=
    \sum_{j: \sigma\in H_1\cdot\theta_j}\frac{k_j}{||\tau||} =
    \sum_{j: \sigma\in H_1\cdot\theta_j}\frac{||\sigma_j||}{||\tau||}\cdot \frac{1}{|H_1\cdot\theta_j|\cdot ||\theta_j||}.$$
    By (4), $\frac{||\sigma_j||}{||\tau||}$ takes only finitely many values. We also know that $|H_1\cdot \theta_j|\le |\Irr(H)|$ and $||\theta_j||\le |H|$.
    The last sum above has at most $m\le [G:H_1]$ summands and each summand takes finitely many values, thus $\frac{\lambda_\sigma}{||\tau||}$ takes only finitely many values. Let us denote the set of all possible values of this fraction by $S$. Now we can take $\aleph':=\sum\limits_{\sigma\in\Irr(H)} S\cdot\sigma\subset \mathbb{Q}_{\ge 0}\rp(H)$. We know that for each $\tau \in \Irr(G)$ there exists some $b\in \aleph'$ such that $i^*(\tau)=||\tau|| \cdot b$. To construct $\aleph$ we can take for each $b\in \aleph'$ a vector $a\in \rp(H)$ proportional to $b$ with the greatest common divisor of the coordinates equal to $1$. This completes the proof.  
    
\end{proof}

\subsection{Flexible stability and almost normal subgroups}

\begin{definition}
    We will call a finitely generated cone $\cone$ with a fixed set $A_\cone \subset \mathbb{Z}\cone$ (not necessarily from $\cone$) primitive if $\cone=\bigoplus_{a\in A_\cone} \cone_a$, where $\cone_a:=\cone\cap \mathbb{Z}a\neq \{0\}$. 
\end{definition}

This means that each $\xi\in\cone$ has a unique decomposition
$$\xi=\sum_{a\in A_\cone}\xi_a,$$
where $\xi_a\in\cone_a$.
Note that $A_\cone$ is necessarily linearly independent and generates $\latt$. By the coordinates of an element from $\latt$ we will mean the coordinates with respect to $A_K$.

For a primitive cone $\cone=\bigoplus_{a\in A_\cone}\cone_a$ and any $a\in A_\cone$ we will fix $k_a:=\min\{k\in\mathbb{Z}_{>0}:ka\in \cone\}$ and define $\widehat{\cone}:=\bigoplus_{a\in A_\cone} \mathbb{Z}_{\ge 0}k_aa\subseteq \cone$. 
Note that $\rcone=\mathbb{R}_{\ge 0}\widehat{\cone}$ and $\mathbb{Z}\widehat{\cone}\cap \mathbb{R}_{\ge 0}\widehat{\cone}=\widehat{\cone}$, therefore, one can view $\widehat{\cone}$ as a positive cone in $\rlatt$.

Let $d\colon\latt \to \mathbb{Z}^m$ be some linear map. 
Lemma 5.2 from \cite{lazarovich2021virtually} gives a nice bound for the distance between $\xi\in \rcone$ and some element from $\widehat{\cone}\cap\ker d$. Lemma 5.3 from \cite{lazarovich2021virtually} uses the fact that $||d(\xi)||\succeq 1$ for $\xi\in\widehat{\cone}\setminus\ker d$, which is also true for $\xi\in\cone\setminus \ker d$. This allows us to reformulate this result as following.

\begin{lemma}\label{lem_laz}
   Let $\cone=\bigoplus_{a\in A_\cone} \cone_a$ be a primitive cone and let $d\colon\latt\to\mathbb{Z}^m$ be some linear map. Then for any $\xi\in\cone\setminus \ker d$ there exists $v\in \widehat{\cone}\cap \ker d$ such that 
    $$||\xi-v||\preceq ||d(\xi)||, \quad ||v||\le||\xi||.$$
\end{lemma}

\begin{lemma}\label{technical_for_flex}
   Let $\cone=\bigoplus_{a\in A_\cone} \cone_a$ be a primitive cone and let $d\colon\latt\to\mathbb{Z}^m$ be some linear map. Then for any $\xi\in \cone$ there exist $\xi'\in \cone$ and $\eta\in\widehat{\cone}$ such that either $\xi'_a=\xi_a$, or $\xi'_a=0$ for $a\in \cone_A$, and 
    $$||\xi-\xi'||,||\eta||\preceq||d(\xi)||,\quad d(\xi'+\eta)=0.$$
\end{lemma}

\begin{proof}
    If $d(\xi)=0$, we take $\xi':=\xi$ and $\eta:=0$. Otherwise, 
    for each $a\in A_\cone$ when possible we fix some vector $\zeta^a\in \widehat{\cone}\cap\ker d$ such that $(\zeta^a)_a\ne 0$. We take $v\in\widehat{\cone}\cap\ker d$ from Lemma~\ref{lem_laz} and define 
    $\xi'$ by saying that $\xi'_a=0$ if $v_a=0$ and $\xi'_a=\xi_a$ otherwise. Note that $||\xi'-v||\le||\xi-v||$.
    
    We can add to $v$ a sum of vectors $\zeta^a$ of the norm at most $C||\xi-v||$ for some constant $C$ to increase each coordinate enough so that the result will be of the form $\xi'+\eta$ with $\xi'+\eta\ge \xi'+w_0(\widehat{\cone}\cap\ker d)$, where $w_0(\widehat{\cone}\cap\ker d)$ is from Lemma~\ref{dense}. It follows that $\eta\in \widehat{\cone}$.  
\end{proof}

For an amalgamated free product $G_1*_HG_2$ given by inclusions $i_j\colon H\to G_j$ with $i_j(H)$ finite and almost normal in $G_j$, one can consider $\aleph_j:=\aleph(i_j)$ defined in Proposition~\ref{about_A0}.  
We define the maps 
$$p\colon \Lambda(G_1)\oplus\Lambda(G_2)\to \mathbb{Z}\aleph_1\oplus\mathbb{Z}\aleph_2,\quad p(\tau_1,\tau_2)=(k_1\cdot a_1,k_2\cdot a_2),$$
$$d\colon \mathbb{Z}\aleph_1\oplus\mathbb{Z}\aleph_2\to \Lambda(H), \quad 
d(a_1,a_2)= a_1-a_2,$$
where $i_j^*(\tau_j)=k_j a_j$ for $\tau_j\in \Irr(G_j)$ and $a_j\in \aleph_j$. Note that $d\circ p =i_1^*-i_2^*$ and the cone $\cone:=p\left(\rp(G_1)\oplus\rp(G_2)\right)$ is primitive. 

Given $\xi_j\in\rp(G_j)$, one can apply Lemma~\ref{technical_for_flex} to $p(\xi_1,\xi_2)$ and get $\xi'=(\xi_1',\xi_2')$ and $\eta=(\eta_1,\eta_2)$. By construction and the properties of the map $p$ these $\xi_j'$ and $\eta_j$ correspond to the elements from $\rp(G_j)$ that satisfy conditions of Theorem~\ref{main_flex} (the zeroing of the coordinates of $p(\xi_1,\xi_2)$ corresponds to the zeroing of the coordinates of $\xi_j$) and, therefore, give us the following result. 

\begin{theorem}\label{amalgam_flex}
    An amalgamated free product of two flexibly $\mathu$-stable groups over an almost normal finite subgroup is flexibly $\mathu$-stable. 
\end{theorem}

For an HNN extension $G*_{i_1,i_2}$ given by two inclusions $i_1,i_2\colon H\to G$ with $i_j(H)$ finite and almost normal in $G$ one can consider $\aleph_j:=\aleph(i_j)$. We will denote the least common multiple of $b_1,b_2\in\mathbb{Z}$ by $LCM(b_1,b_2)$. For $a_1\in\aleph_1$ and $a_2\in\aleph_2$ we take $l_1=l_1(a_1,a_2):=\frac{LCM(||a_1||,||a_2||)}{||a_1||}$ and $l_2=l_2(a_1,a_2):=\frac{LCM(||a_1||,||a_2||)}{||a_2||}$.  
This allows us to define the maps 
$$p\colon\Lambda(G)\to\mathbb{Z}(\aleph_1\times\aleph_2),\quad 
p(\tau)=\frac{k_1}{l_1}\cdot(a_1,a_2)=\frac{k_2}{l_2}\cdot(a_1,a_2),$$
$$d\colon \mathbb{Z}(\aleph_1\times\aleph_2) \to\Lambda(H),\quad
d((a_1,a_2))=l_1a_1-l_2a_2,$$
where $i_1^*(\tau)=k_1a_1$ and $i_2^*(\tau)=k_2a_2$. Note that $d\circ p=i_1^*-i_2^*$ and the cone $\cone:=p(\rp(G))$ is primitive. 

\begin{theorem}\label{hnn_flex}
HNN extension of a flexibly $\mathu$-stable group $G$ over two almost normal finite subgroups is flexibly $\mathu$-stable.
\end{theorem}
\begin{proof}
    Given $\varphi\colon G*\mathbb{Z}\to \Aut(X)\in\mathu$ satisfying $d_H(i_1^*(\varphi),i_2^*(\varphi)^{\varphi(t)})\le\delta$, one can apply Lemma~\ref{technical_for_flex} to $p(\varphi_G^\#)$ and get $\xi'$ and $\eta$. We can always add the multiple of the class of the trivial homomorphism to $\eta$ so that $||\xi'||+||\eta||\ge ||\xi||$. We see that $\xi'$ corresponds to some subhomomorphism $\varphi^0\colon G\to \Aut(Y)$ for $Y\subseteq X$ and $\eta$ corresponds to $\varphi^1\colon G\to \Aut(Z)$ for some $Z\in\Cat$. 
    By construction
    $$|X|-|Y|,|Z|\preceq ||i_1^*(\varphi_G^\#)-i_2^*(\varphi_G^\#)||\preceq \delta^\frac{1}{s}\cdot |X|,$$
    hence, for small enough $\delta$ we also have $|X|\preceq |Y|$. 
    Note that $\left(\varphi^{(\varphi(t)|_Y)\sqcup 1_{Y^\perp}}\right)|_Y=(\varphi^0)^{\varphi(t)|_Y}$, where $X=Y\coprod Y^\perp$. 
    We also have $d\left((\varphi(t)|_Y)\sqcup 1_{Y^\perp},\varphi(t)\right)\preceq \left(\frac{|X|-|Y|}{|X|}\right)^s\preceq\delta$ by Lemma~\ref{lem_restr_useful}. 
    
    Applying Lemma~\ref{property_restr} we get 
    $$d_H\left(i_1^*(\varphi^0),i_2^*(\varphi^0)^{\varphi(t)|_Y}\right)\preceq d_H\left(i_1^*(\varphi),i_2^*(\varphi)^{(\varphi(t)|_Y)\sqcup1_{Y^\perp}}\right)\left(\frac{|X|}{|Y|}\right)^s+\left(\frac{|X|-|Y|}{|Y|}\right)^s\preceq$$
    $$\preceq \left(d_H(i_1^*(\varphi),i_2^*(\varphi)^{\varphi(t)})+2d((\varphi(t)|_Y)\sqcup 1_{Y^\perp},\varphi(t))\right) \left(\frac{|X|}{|Y|}\right)^s+\delta\preceq \delta.$$    
    We define $\psi:=\varphi^0\sqcup \varphi^1\colon G\to \Aut(Y\coprod Z)$ and identify $X$ with a subobject of $Y\coprod Z$. 
    Restricting to $Y$ and applying Lemma~\ref{property_restr} and Lemma~\ref{lem_double_restr} we get that $d_S(\psi|_X,\varphi)\preceq\delta$ and also that 
    $$d_H\left(i_1^*(\psi), i_2^*(\psi)^{(\varphi(t)|_Y)\sqcup 1_Z}\right)\preceq \delta.$$
    The last inequality allows us to extend $\psi$ to $G*\mathbb{Z}$ by applying the conjugator argument as in Theorem~\ref{operator_HNN} to $i_1^*(\psi)$ and $i_2^*(\psi)^{(\varphi(t)|_Y)\sqcup 1_Z}$. Again by Lemma~\ref{property_restr} we get 
    $d(\psi(t)|_X,\varphi(t))\preceq \delta.$
    This finishes the proof. 
\end{proof}

\section{Decomposing maps and stability}
\subsection{Main property of decomposing maps} 
Let $K=\bigoplus_{a\in A_K} K_a$ be a primitive cone. We assume that $\mathbb{Z} K$ is equipped with some weighted $l_1$-norm $||\cdot||$ with respect to the generating set $A_K$.
\begin{definition}
\label{def:decomposing_map}
    We will say that a homomorphism $p\colon \Lambda(G)\to \mathbb{Z}K$ is a decomposing map if the following conditions are satisfied:
    \begin{enumerate}
   \item  For any $\tau\in \Irr(G)$ there exists $ a\in A_K$ such that  $p(\tau)\in K_a$ and $||\tau||=||p(\tau)||$.
   \item $p(\rp(G))=K.$
    \item For any homomorphisms $ \varphi_1,\varphi_2\colon G\to \Aut(X)$ we have $||p(\varphi_1^\#) - p(\varphi_2^\#)||\preceq d_S(\varphi_1,\varphi_2)^{\frac{1}{s}}|X|.$
    \end{enumerate}
    
\end{definition}

Note that it follows from condition (1) that $||\eta||=||p(\eta)||$ for any $\eta\in\rp(G)$. 
Also note that condition (1) implies that for any homomorphism $\varphi\colon G\to \Aut(X)$ there is a decomposition $X=\coprod_{a\in A} X_a$, $\varphi=\coprod_{a\in A} \varphi_a$, where for each $\varphi_a\colon G\to \Aut(X_a)$ we have $p(\varphi_a^\#)\in K_a$. Condition (2) is inconsequential, since one can always take $p(\rp(G))$ as $K$ (it is primitive by (1)).   

\begin{prop}\label{prop_stable}
    Let $G$ be a $\mathu$-stable group and let $p\colon\Lambda(G)\to \mathbb{Z}K$ be a decomposing map. 
    Then for any $\varepsilon>0$ there exists $\delta>0$ such that for any homomorphism $\psi\colon G\to \Aut(X)\in\mathu$ and any element $\xi\in p(\rp(G))$ satisfying 
    $$||\xi||=||\psi^\#||, \quad ||\xi-p(\psi^\#)||\le\delta ||\xi||$$ 
    there exists a homomorphism $\psi'\colon G\to \Aut(X)$ that satisfies $$p\left((\psi')^\#\right)=\xi,\quad d_{S}(\psi,\psi')\le\varepsilon.$$
\end{prop}

Informally speaking, this proposition means that any possible small change of $p(\psi^\#)$ corresponds to some small change of $\psi$.

\begin{proof}
    If $\xi=p(\psi^\#)$, we will take $\psi'=\psi$. Otherwise, we can assume that $1\le||\xi-p(\psi^\#)||\le\delta||\xi||$. 
    
    Let us fix an element $w_0=w_0(K)$ from Lemma~\ref{dense}. Consider the generating set $A_\cone$ for $\latt$ and operations $\min$, $\max$ and $\le$ with respect to it.  Take $\zeta:=\max(\overline{0}, \min(\xi, p(\psi^\#))-w_0)$. Note that $\zeta\le \xi, p(\psi^\#)$ and $||\xi-\zeta||, ||p(\psi^\#)-\zeta||\le \delta||\xi|| + ||w_0||\preceq \delta ||\xi||$. Another important property of $\zeta$ is that for any $\xi'\le\zeta$ the difference $\xi-\xi'$ has each coordinate either equal to the corresponding coordinate of $\xi$, or not less than the corresponding coordinate of $w_0$, so by primitivity of $\cone$ we have $\xi-\xi'\in\cone$.

    Let us consider the decomposition 
    $X=\coprod_{a\in A_K} X_a$.
    For each $a\in A_K$ we will fix a subobject $Y_a\subseteq X_a$ with $|Y_a|=||\zeta_a||$ and consider the restriction $\varphi^a:=\psi|_{Y_a}\colon F_S\to \Aut(Y_a)$. Note that by construction $|X_a|-|Y_a|\preceq \delta|X|$.
    For some $\kappa_a\le 2$ there exist actual homomorphisms $\gamma^a\colon G\to \Aut(Y_a)$ satisfying $d_S(\varphi^a,\gamma^a)\le \kappa_a$. We assume that $\kappa_a$ are the minimal constants with this property. 
    For each $a$ let us consider the components $(\gamma^a)_a\colon G\to \Aut(Z_a)$, where $Z_a\subseteq Y_a$ and take $\alpha:=\coprod_a (\gamma^a)_a$. 
    Denote $Z:=\coprod_aZ_a$ and $Y:=\coprod_a Y_a$.
    Since $p(\alpha^\#)\le \zeta$, there exists some $\beta\colon G\to \Aut\left(Z^{\perp}\right)$ such that $p(\alpha^\#+\beta^\#)=\xi$, and we define $\psi':=\alpha\sqcup\beta$.
    
    Our goal is to estimate $d_S(\psi,\psi')$. Since $d_S(\varphi^a,\gamma^a)\le\kappa_a$, we have by Lemma~\ref{property_sums}
    $$d_S\left(\coprod_a\varphi^a,\coprod_a\gamma^a\right)\preceq \sum_a \kappa_a\cdot\left(\frac{|Y_a|}{|Y|}\right)^s.$$
    Note that $\psi|_Y=\coprod\varphi^a$ (see Remark~\ref{rem:good_restriction}) and $\psi'|_Z=\coprod (\gamma^a)_a$. By Lemma~\ref{lem_double_restr} and Lemma~\ref{property_restr} we have
    $$d_S(\psi|_Z,\psi'|_Z)\preceq d_S\left(\coprod_a\varphi^a|_{Z_a},\coprod_a(\gamma^a)_a\right) +
    \left(\frac{|X|-|Z|}{|Z|}\right)^s
    \preceq  \sum_a\kappa_a\cdot\left(\frac{|Y_a|}{|Z|}\right)^s+
    \left(\frac{|X|-|Z|}{|Z|}\right)^s.$$
    
    The hard part is to estimate $|X|-|Z|$ or, equivalently, to estimate $|X_a|-|Z_a|$ for each $a$. Observe that by Lemma~\ref{property_restr}
    $$d_S\left(\psi|_{X_a},\gamma^a\sqcup 1_{X_a\cap Y_a^\perp}\right)\preceq \kappa_a\left(\frac{|Y_a|}{|X_a|}\right)^s+\left(\frac{|X_a|-|Y_a|}{|X_a|}\right)^s\preceq\kappa_a+\left(\delta\frac{|X|}{|X_a|}\right)^s.$$
    By property (3) of decomposing maps we have
    $$||p((\psi|_{X_a})^\#)-p((\gamma^a\sqcup1_{X_a\cap Y_a^\perp})^\#)||\preceq (\kappa_a)^{\frac{1}{s}} |X_a| +\delta |X|.$$ 
    We know that $||p((\psi|_{X_a})^\#)_a||=|X_a|$ and $||p((\gamma^a\sqcup1_{X_a\cap Y_a^\perp})^\#)_a||$ is either equal to $|Z_a|$, or to $|Z_a|+|X_a|-|Y_a|$. Hence,
    $$|X_a|-|Z_a|\preceq (\kappa_a)^{\frac{1}{s}} |X_a| +\delta |X|.$$
    
    Using Lemma~\ref{property_restr} we conclude
    $$d_S(\psi,\psi')\preceq d_S(\psi|_Z,\psi'|_Z)\cdot\left(\frac{|Z|}{|X|}\right)^s +\left(\frac{|X|-|Z|}{|X|}\right)^s\preceq 
    \sum_a\kappa_a\cdot\left(\frac{|Y_a|}{|X|}\right)^s+\left(\frac{|X|-|Z|}{|X|}\right)^s=$$
    $$=\sum_a\kappa_a\cdot\left(\frac{|Y_a|}{|X|}\right)^s+
    \left(\frac{\sum_a(|X_a|-|Z_a|)}{|X|}\right)^s
    \preceq \delta^s+ \sum_a\kappa_a\cdot\left(\frac{|X_a|}{|X|}\right)^s.$$
    This means that for some $C>0$ we have 
    $$d_S(\psi,\psi')\le C\left( \delta^s+ \sum_a\kappa_a\cdot\left(\frac{|X_a|}{|X|}\right)^s\right).$$
    
    Finally, we are going to estimate $\kappa_a$. Since $G$ is stable, we can fix $\theta>0$ and a finite set $E\subseteq \ker \pi$ (where $\pi\colon F_S\to G$) such that for any homomorphism $\varphi\colon F_S\to \Aut(W)\in\mathu$ satisfying $d_E(\varphi,1_W)\le\theta$
    there exists a homomorphism $\gamma\colon G\to \Aut(W)$ that satisfies $d_S(\varphi,\gamma)\le \frac{\varepsilon}{2C|A_K|}$. 
    By Lemma~\ref{lem_almost_hom} we have $d_E(\varphi^a,1_{Y_a})\le C'(E)\cdot\left(\frac{|X_a|-|Y_a|}{|X_a|}\right)^s$. 
    For $a\in A_K$ we have two possibilities:
    \begin{enumerate}
       \item either $C'\left(\frac{|X_a|-|Y_a|}{|X_a|}\right)^s\le\theta$ and we have $C\kappa_a\cdot\left(\frac{|X_a|}{|X|}\right)^s\le \frac{\varepsilon}{2|A_K|}\cdot 1$,
       \item or $\theta<C'\cdot\left(\frac{|X_a|-|Y_a|}{|X_a|}\right)^s\le C''\left(\frac{\delta|X|}{|X_a|}\right)^s$ and we have $C\kappa_a\cdot\left(\frac{|X_a|}{|X|}\right)^s\le 2C\cdot\left(\frac{|X_a|}{|X|}\right)^s <2C C''\frac{\delta^s}{\theta}$. 
       \end{enumerate}
       Here $C'$ and $C''$ depend on $\varepsilon$. For small enough $\delta$ both inequalities $C\delta^s\le\frac{\varepsilon}{2}$ and $2C C''\frac{\delta^s}{\theta}\le \frac{\varepsilon}{2|A_K|}$ hold, hence
    $$d_S(\psi,\psi')\le \frac{\varepsilon}{2}+|A_K|\cdot \frac{\varepsilon}{2|A_K|}=\varepsilon.$$
\end{proof}

\subsection{Examples of decomposing maps and stability results}

Now let us consider some examples of decomposing maps. 

Let us consider a finite group $H$ and an inclusions $i\colon H\to G$. If $i(H)$ is normal in $G$, then the set $\aleph(i)=:A_K$ can be taken linearly independent. This gives us the decomposing map $p\colon \Lambda(G)\to \mathbb{Z}K=\mathbb{Z}\aleph(i)$, where $K=i^*(\rp(G))$. Indeed, conditions (1) and (2) are clear and condition (3) follows from Lemma~\ref{distances_connection}. 

For an amalgamated free product $G_1*_HG_2$, if $i_j(H)$ are finite and normal in $G_j$, then the map $p\colon \Lambda(G_1)\oplus\Lambda(G_2)\to \mathbb{Z}\aleph_1\oplus\mathbb{Z}\aleph_2$ defined in Section 5.2 is a direct sum of two decomposing maps $p_1$ and $p_2$. 

\begin{theorem}\label{amalgam_stability}
    An amalgamated free product of two $\mathu$-stable groups $G_1$ and $G_2$ over a normal finite subgroup $H$ is $\mathu$-stable.
\end{theorem}

\begin{proof}
    Let us consider $\cone_j:=i_j^*(\rp(G_j))$.  
    Any fixed $\delta>0$ and any homomorphism $\varphi\colon G_1*G_2\to \Aut(X)\in\mathu$ satisfying $d_{E_0}(\varphi, 1)\le \delta$ give us the element $\xi=p\left(\varphi_1^\#,\varphi_2^\#\right)\in \cone:=\cone_1\oplus\cone_2$ that satisfies $||d(\xi)||\preceq\delta^s|X|$ (by Lemma~\ref{distances_connection}) for a map $d\colon \latt\to\Lambda(H)$ defined in Section 5.2.
    
    If $\xi\in\ker d$, then we take $\psi:=\varphi$. Otherwise, we take $v=(v_1,v_2)\in\widehat{K}\cap \ker d$ from Lemma~\ref{lem_laz}. Since $(1,1)\in \ker(d)$, we may assume that $||v||=||\xi||$ and  apply Proposition~\ref{prop_stable} to $v_1$ and $v_2$. 
    
    In both cases we get $\psi_j\colon G_j\to \Aut(X)$ such that 
    $$i_1^*(\psi_1^\#)=i_2^*(\psi_2^\#), \quad d_H(i_1^*(\psi_1),i_2^*(\psi_2))\le \varepsilon,$$
    where $\varepsilon>0$ can be taken arbitrary small for small enough $\delta$. The conjugator argument finishes the proof.    
\end{proof}

For an HNN extension $G*_{i_1,i_2}$ consider the map $p\colon\Lambda(G)\to\mathbb{Z}(\aleph_1\times\aleph_2)$ defined in Section 5.2. It satisfies conditions (1) and (3) from Definition~\ref{def:decomposing_map} and gives a decomposing map after changing the codomain.   

\begin{theorem}\label{hnn_stability}
    An HNN extension of a $\mathu$-stable group $G$ over normal finite subgroups $i_1(H)$ and $i_2(H)$ is $\mathu$-stable. 
\end{theorem}
\begin{proof}
    Let us fix $\delta>0$. For any homomorphism $\varphi\colon G*\mathbb{Z}\to \Aut(X)\in \mathu$ that satisfies ${d_{H}(i_1^*(\varphi),i_2^*(\varphi)^{\varphi(t)})\le\delta}$ we have $\xi:=p(\varphi_G^\#)\in \cone$ satisfying $||d(\xi)||\preceq \delta^s|X|$. 

    If $\xi\in\ker d$, we take $\psi:=\varphi$. Otherwise, we take $v\in\widehat{K}\cap \ker d$ from Lemma~\ref{lem_laz}. Since $(1,1)\in \ker(d)$, we may assume that $||v||=||\xi||$ and apply Proposition~\ref{prop_stable} to $v$. 
    
    In both cases we get $\psi\colon G\to \Aut(X)$ such that 
    $$i_1^*(\psi^\#)=i_2^*(\psi^\#), \quad d_H\left(i_1^*(\psi),i_2^*(\psi)^{\varphi(t)}\right)\le \varepsilon,$$
    where $\varepsilon>0$ can be taken arbitrary small for small enough $\delta$. And the conjugator argument again finishes the proof.
\end{proof}






\section{Further remarks and observations}

\subsection{Amalgamated free products over ``small'' groups} 

\begin{definition}
    We will say that a group $G$ is $\mathu$-approximable if for any $\varepsilon>0$ and for any finite set $F\subseteq G$ there exists $X\in\Cat$ and a map $\varphi\colon G \to \Aut(X)\in\mathu$ such that
    $$\forall~g,h\in F~~ d(\varphi(g)\varphi(h),\varphi(gh))\le\varepsilon,$$
    $$\forall~e\ne g\in F~~d(\varphi(g),1_X)\ge 2^{1-s}-\varepsilon.$$
\end{definition}
The constant $2^{1-s}$ comes from the fact that there are sufficiently many different elements at this distance in $\Aut(X)$. $\mathu_P$-approximable groups are called sofic and $\mathu_{HS}$-approximable groups are called hyperlinear. At this point there are no known examples of groups that are not $\mathu$-approximated.   

The following theorem is well-known. See Theorem 2 in \cite{glebsky2009almost} for $\mathu_P$. For $\mathu_{HS}$ it follows from Mal'cev's theorem.

\begin{theorem}\label{stable_are_rf}
    If a finitely generated group is $\mathu$-approximable and $\mathu$-stable, then it is residually finite.
\end{theorem}

\begin{lemma}\label{res_fin_property}
    If $i\colon H\to G$ is an inclusion of a finite group $H$ into a residually finite group $G$, then there exists $v\in i^*(\rp(G))$ such that $v_\sigma\ne 0$ for any $\sigma\in \Irr(H)$.
\end{lemma}
\begin{proof}
    Let us take a finite index normal subgroup $N\triangleleft G$ such that $N\cap i(H)=\{e\}$. One can consider $H$ as a subgroup of $G/N$. 
    For $\mathu_{HS}$ one can take $l^2(G/N)$, its restriction to $H$ contains a copy of $l^2(H)$, therefore the condition is satisfied. 
    For $\mathu_P$ one can consider a disjoint union of actions of $G$ on $(G/N)/H_0$, where $H_0$ goes over subgroups of $H$.  
\end{proof}

Note that this lemma does not imply even that $\mathbb{R}(i^*(\rp(G)))=\mathbb{R}(\Lambda(H))$.

\begin{theorem}\label{2dim} 
Assume that finitely generated groups $G_1$ and $G_2$ are $\mathu$-approximable and $\mathu$-stable. If $H$ is a finite group with $\rk(\Lambda(H))=2$, then for any inclusions $i_1\colon H \to G_1$ and $i_2\colon H \to G_2$ the corresponding amalgamated free product $G_1*_{H}G_2$ is flexibly $\mathu$-stable.
\end{theorem}

\begin{proof}
    It is sufficient to check that conditions of Corollary~\ref{cor_flex} are satisfied. Given $\xi_1\in \cone_1:=i_1^*(\rp(G_1))$ and $\xi_2\in \cone_2:=i_1^*(\rp(G_1))$, one can consider the projection $p\colon\rp(H)\to\rp(H)$ onto the second coordinate subspace. Note that group $G_1$ and $G_2$ are residually finite by Theorem~\ref{stable_are_rf}, so $p(\cone_1)$ and $p(\cone_2)$ are not trivial by Lemma~\ref{res_fin_property}. One-dimensional cones are always finitely generated. Let us fix finite sets $V_1\subset\cone_1$ and $V_2\subset\cone_2$ such that $p(V_1)$ and $p(V_2)$ generate $p(\cone_1)$ and $p(\cone_2)$ respectively. 

    There exist $\eta_1'\in p(\cone_1)$ and $\eta_2'\in p(\cone_2)$ such that $p(\xi_1)+\eta_1'=p(\xi_2)+\eta_2'$ and $||\eta_j'||\preceq ||p(\xi_1)-p(\xi_2)||\le ||\xi_1-\xi_2||$. We can construct $\eta_j\in\cone_j$ such that $p(\eta_j)=\eta_j'$ as sums of elements from $V_j$. Note that $||\eta_j||\preceq ||\eta_j'||$. The vectors $\xi_1+\eta_1$ and $\xi_2+\eta_2$ have the same second coordinates and $||(\xi_1+\eta_1)-(\xi_2+\eta_2)||\preceq ||\xi_1-\xi_2||$, hence we can add to $\eta_1$ or $\eta_2$ some multiple of $1=(1,0)\in\cone_j$ and get the required $\eta_1$ and $\eta_2$. 
\end{proof} 

\begin{remark}
    Theorem~\ref{2dim} also holds for flexibly stable $G_1$ and $G_2$, since Theorem~\ref{stable_are_rf} does.  
\end{remark}

\begin{remark}
    The list of finite groups $H$ satisfying $\rk(\Lambda(H))= 2$ contains $\mathbb{Z}_2$ for $\mathu_{HS}$ and $\mathbb{Z}_p$ for $\mathu_P$, where $p$ is prime. We note that one can slightly expand the result of Theorem~\ref{2dim} using the multiplicative structure and $\lambda$-operations on $\Lambda(H)$.    
\end{remark}

\begin{exmp}
    Let us take two operator norm stable crystallographic groups $G_1$ and $G_2$ $($see \cite{eilers2020c}$)$ with a subgroup $\mathbb{Z}_2$. It follows from Theorem~\ref{operator_amalgam} the $G_1*_{\mathbb{Z}_2}G_2$ is operator norm stable and from Theorem~\ref{2dim} that it is flexibly $HS$-stable. 
    Note that for amenable groups operator norm stability implies $HS$-stability. It is not known if it is true in general $($see Question 4 in \cite{eilers2020c}$)$.
    If $G_1*_{\mathbb{Z}_2}G_2$ is not $HS$-stable, this will give a negative answer to this question.  
\end{exmp}

\subsection{The homomorphism alteration property} 
Another approach to proving stability of amalgamated free products of stable groups was introduced for $\mathu_P$ in the recent paper \cite{arzhantseva2023constraint} by Arzhantseva and P{\u{a}}unescu. Let us briefly discuss their idea using our terminology.

Given $\varphi\colon G_1*G_2\to \Aut(X)$ for which $d_{S_H}(i_1^*(\varphi),i_2^*(\varphi))$ is close to $0$, one wants to construct a homomorphism $\psi\colon G_1*_HG_2\to \Aut(X)$ that is close to $\varphi$. The idea is to look for $\psi$ that satisfies $\psi_1=\varphi_1$. It is equivalent to looking for $\psi_2\colon G_2\to\Aut(X)$ that is close to $\varphi_2\colon G_2\to\Aut(X)$ and satisfies $i_2^*(\psi_2)=i_1^*(\varphi)$.  

We will use $\varepsilon-\delta$ terminology instead of using sequences of maps and ultraproducts. The following notion is equivalent to being ``$\psi_0$-constraint stable for any $\psi_0$'' as in \cite{arzhantseva2023constraint}. 

\begin{definition}\label{property_of_a_pair}
    We will say that a pair of groups $(G,H)$ with an inclusion $i\colon H\to G$ satisfy (homomorphism) alteration property with respect to $\mathu$ if for any $\varepsilon>0$ there exists $\delta>0$ such that for any homomorphisms $\varphi\colon G\to \Aut(X)\in\mathu$ and $\psi_0\colon H\to \Aut(X)$ satisfying $d_{S_H}(\psi_0, i^*(\varphi))\le \delta$ there exists $\psi\colon G\to \Aut(X)$ such that $d_{S_G}(\psi, \varphi)\le \varepsilon$ and $i^*(\psi)=\psi_0$.
\end{definition}

Note that this property requires so-called (homomorphism) extension property (see \cite{arzhantseva2023constraint}). In our notations extension property can be written as $i^*(\rp(G))=\rp(H)$. Let us mention that if $G$ and $H$ are finite then alteration property for (G,H) is equivalent to extension property. 

\begin{lemma}\label{lem_for_finite_equiv}
    Let $G$ be a finite group and $i\colon H\to G$ be an inclusion. Then $(G,H)$ has alteration property if and only if it has extension property. 
\end{lemma}
\begin{proof}
    We only need to prove that extension property implies alteration property. Given $\varphi\colon G\to \Aut(X)$ and $\psi_0\colon G\to\Aut(X)$, consider $\xi:=\min(i^*(\varphi^\#),\psi_0^\#)$. 
    
    Note that for any $\eta\in \rp(H)$ and any $\zeta\in\rp(G)$ such that $i^*(\zeta)\ge \eta $ there exists $\zeta'\le\zeta$ such that 
    $$i^*(\zeta')\ge \eta, \quad ||\zeta'||\preceq ||\eta||.$$
    To prove it we start with $\zeta'=0$ and on each step add to it an element of $\Irr(G)$ making $i^*(\zeta')$ ``cover'' more of $\eta$ (formally we increase $||\min(\eta,i^*(\zeta'))||$).

    Now we can apply this fact to $\eta=i^*(\varphi^\#)-\xi$ and $\zeta=\varphi^\#$ and get $\zeta'$. This $\zeta'$ corresponds to some $\varphi(G)$-invariant subobject $Y\subseteq X$. By construction, $i^*(\varphi|_Y^\#)\le \psi_0^\#$, hence by extension property there exists $\alpha\colon G\to \Aut(Y^\perp)$ such that $i^*(\varphi|_Y\sqcup\alpha)^\#=\psi_0^\#$. The conjugator argument and standard inequalities finish the proof. 
\end{proof}

The following proposition uses only the conjugator argument. 

\begin{prop}\label{with_a_property_of_a_pair}
    If a group $G_1$ is $\mathu$-stable and a pair $(G_2,H)$ has alteration property with respect to $\mathu$, then $G_1*_HG_2$ is $\mathu$-stable.
\end{prop}

First examples (for $\mathu_P$) of group pairs $(G,H)$ with an infinite $G$ that have alteration property were given in \cite{arzhantseva2023constraint}. 

\begin{theorem}[Corollary 5.12 in \cite{arzhantseva2023constraint}]
    Let $G$ be an amenable group and $H\le G$ be a finite subgroup. Then $(G,H)$ has alteration property with respect to $\mathu_P$ if the following conditions hold:
    \begin{enumerate}
        \item $(G,H)$ satisfies extension property;
        \item $Sub(G)$ is countable;
        \item every almost normal subgroup of $G$ is profinitely closed. 
    \end{enumerate}
\end{theorem}

Let us give an example of group pairs that satisfy alteration property with respect to $\mathu_{HS}$.  

\begin{prop}
    Let $G$ be a finite abelian group and $H\le G$ be its subgroup. Then $(G,H)$ has alteration property with respect to $\mathu_{HS}$. 
\end{prop}
\begin{proof}
    By Lemma~\ref{lem_for_finite_equiv} it is sufficient to prove extension property. And this property follows from the fact that for appropriate systems of generators groups $G$ and $H$ can be written as
    $$H= \bigoplus_i\mathbb{Z}/m_i\mathbb{Z}\subseteq\bigoplus_i\mathbb{Z}/n_i\mathbb{Z}=G ,$$
    where $m_i$ divides $n_i$ for each $i$.  
\end{proof}

Observe that some of these pairs have no alteration property with respect to $\mathu_P$ (for example, $(\mathbb{Z}/4\mathbb{Z},\mathbb{Z}/2\mathbb{Z})$ has no extension property with respect to $\mathu_P$).

\subsection{Very flexible stability}
In this paper we used the restriction construction since it was crucial in the proof of Proposition~\ref{prop_stable}. But there is also a different approach. One can define the distance between $a_1 \in \Aut(Y_1)$ and $a_2\in \Aut(Y_2)$ over a common subobject $X\subseteq Y_1,Y_2$ for $\mathu_{HS}$ as 
$$d_{X}(a_1,a_2):=||p_1a_1p_1-p_2a_2p_2||_{HS,X},$$
where $p_1$ and $p_2$ are projections onto $X$ from $Y_1$ and $Y_2$.  
For $\mathu_P$ we define it as 
$$d_{X}(a_1,a_2):=\#\{x\in X: a_1(x)\ne a_2(x)\}\cdot\frac{1}{|X|}.$$

This approach gives another reformulation of flexible stability and very flexible stability both for $\mathu_{HS}$ and $\mathu_P$.
One can also slightly modify the conjugator argument (Proposition~\ref{prop_conj}) to make it applicable to this approach.

\begin{prop}\label{new_conjugator}
    For a finite group $H$, any two homomomorphisms $\varphi_j\colon H\to\Aut(Y)$ satisfying $\varphi_1^\#=\varphi_2^\#$ and any $X\subseteq Y$ there exists $t\in\Aut(Y)$ such that
    $$d_X(t,1_X)\preceq d_{X,H}(\varphi_1,\varphi_2),\quad \varphi_2^{t}=\varphi_1.$$
\end{prop}

Note that Proposition~\ref{new_conjugator} allows one to generalise Theorem~\ref{double_HNN}, Theorem~\ref{double_amalgam}, Theorem~\ref{amalgam_flex} and Theorem~\ref{hnn_flex} to very flexible stability.   
That means, if we assume starting groups to be very flexibly $\mathu$-stable, then the resulting group will also be very flexibly $\mathu$-stable.

\subsection{Open questions}

We conclude this paper with several open questions. 

Let $G$ be a group and $i\colon H\to G$ be an inclusion of a finite group $H$. What can one say about $i^*(\rp(G))$?  In particular, we are interested in the following question.  
\begin{question}
    Is $i^*(\rp(G))$ always finitely generated (as a semigroup)? 
\end{question}

In principle, the answer might be different for actions and for representations. 
Both a positive and a negative answer would be useful for a better understanding of these objects. Note that in our proofs we use stronger properties of $i^*\colon \rp(G)\to \rp(H)$.

The following question arises naturally from our results on flexible stability for $\mathu_{HS}$ and $\mathu_P$.

\begin{question}
    Theorem~\ref{amalgam_flex} gives only flexible $\mathu$-stability for amalgamated free products of $\mathu$-stable groups over almost-normal subgroups. Is it true that these amalgamated free products are in fact always $\mathu$-stable?  
\end{question}

If the answer is negative, it will give us first examples of flexibly $\mathu$-stable but not $\mathu$-stable groups.

As a generalisation of several results provided in this paper, we ask the following question: 

\begin{question}
    Is an amalgamated free product of two $\mathu$-stable groups over an arbitrary finite subgroup (flexibly) $\mathu$-stable? 
\end{question}

For the operator norm one can ask if Theorem~\ref{operator_amalgam} can be generalised in the following way. 

\begin{question}
    Is an amalgamated free product of operator norm stable groups over an infinite central subgroup always operator norm stable?
\end{question}

This question is motivated by the following example. The group $\mathbb{Z}*_{2\mathbb{Z}}\mathbb{Z}$ given by the relation $x^2=y^2$ is a fundamental group of the Klein bottle. It is operator norm stable as one of stable crystallographic groups (see \cite{eilers2020c}).

\Addresses

\end{document}